\documentclass[11pt]{amsart}
\usepackage[T1]{fontenc}
\usepackage{lmodern,microtype,mathtools,amssymb,mathrsfs,booktabs,array,longtable}
\usepackage[
  textheight=20cm,
  textwidth=6.0in,
  centering
]{geometry}
\usepackage{tikz}
\usetikzlibrary{positioning}
\tikzset{
  coxeter/.style={
    x=1cm,
    y=.8cm,
    node distance=6mm,
    every node/.style={inner sep=2pt}
  }
}
\usepackage[colorlinks=true,linkcolor=blue!50!black,citecolor=blue!50!black,urlcolor=blue!50!black]{hyperref}
\hypersetup{pdftitle={Finite subgroups of generalized cactus groups},pdfsubject={Finite-subgroup classification for arbitrary Coxeter systems},pdfkeywords={cactus groups, Coxeter systems, finite subgroups, spherical parabolic subgroups}}
\numberwithin{equation}{section}
\newtheorem{theorem}{Theorem}[section]
\newtheorem{proposition}[theorem]{Proposition}
\newtheorem{lemma}[theorem]{Lemma}
\newtheorem{corollary}[theorem]{Corollary}
\newtheorem{maintheoremA}{Theorem}

\newtheorem{maintheoremB}[maintheoremA]{Theorem}
\theoremstyle{definition}

\theoremstyle{remark}
\newtheorem{remark}[theorem]{Remark}
\allowdisplaybreaks[2]
\title[Classification of finite subgroups of generalized cactus groups]{Classification of finite subgroups of generalized cactus groups}
\author{Junseok Kim}
\address{Department of Mathematics \\
        Technion - Israel Institute of Technology \\
        Technion City, Haifa, Israel, 3200003}
\email{jsk8818@campus.technion.ac.il}		

\begin{document}

\begin{abstract}
We classify the finite groups that embed into generalized cactus groups associated with arbitrary Coxeter systems. Every finite subgroup is conjugate to a subgroup of the generalized cactus group associated with a finite standard parabolic subgroup of the underlying Coxeter group. This reduces the classification to explicit lists of finite groups for the finite irreducible Coxeter systems. For types B, D, E, F, and H, we analyze cube stabilizers to obtain these lists. For type $I_2(p)$, we determine the associated generalized cactus group directly from its defining presentation. We explicitly realize every listed group as a subgroup of the corresponding generalized cactus group. When the Coxeter diagram has finitely many vertices, we determine the largest order and all possible orders of finite subgroups.

In type A, we prove directly that a finite group embeds into the ordinary cactus group on $n$ strands if and only if it embeds into a Sylow 2-subgroup of the symmetric group on $n$ letters. We also give weighted embedding criteria in types A and B and counterexamples to Genevois's conjecture on finite subgroups of cactus groups.
\end{abstract}

\maketitle

\section{Introduction}\label{sec:introduction}

For $n\geq2$, the \emph{cactus group} $J_n$ is generated by $s_{p,q}$, where $1\leq p<q\leq n$, with the following relations
\begin{equation}\label{eq:ordinarypresentation}
 \begin{aligned}
  s_{p,q}^2&=1,\\
  {}[s_{p,q},s_{r,t}]&=1
  &&\text{if }[p,q]\cap[r,t]=\emptyset,\\
  s_{p,q}s_{r,t}s_{p,q}
   &=s_{p+q-t,p+q-r}
  &&\text{if }[r,t]\subseteq[p,q].
 \end{aligned}
\end{equation}
Here, $[p,q]=\{p,p+1,\ldots,q\}$. Sending $s_{p,q}$ to the permutation that reverses $[p,q]$ and fixes every element outside this interval defines the canonical epimorphism
\[
 \Sigma_n:J_n\longrightarrow S_n.
\]
Its kernel is called the \emph{pure cactus group}.

Cactus groups play a role in coboundary categories analogous to that of braid groups in braided monoidal categories. In a coboundary category, the commutor induces isomorphisms reversing consecutive tensor factors, and these isomorphisms satisfy the defining relations of $J_n$ \cite[Section~3.1]{HenriquesKamnitzer06}. Henriques and Kamnitzer constructed such a commutor for crystals of a finite-dimensional complex reductive Lie algebra and proved that it defines a coboundary structure \cite[Theorem~6]{HenriquesKamnitzer06}.

Cactus groups also have a topological interpretation. The pure cactus group $\ker\Sigma_n$ is isomorphic to $ \pi_1\bigl(\overline{\mathcal M}_{0,n+1}(\mathbb R)\bigr)$, where $\overline{\mathcal M}_{0,n+1}(\mathbb R)$ is the real locus of the Deligne--Mumford moduli space of stable curves of genus zero with $n+1$ marked points \cite[Theorem~9]{HenriquesKamnitzer06}.

From the viewpoint of geometric group theory, Genevois gave a direct proof that the canonical Cayley graph of $J_n$ is median \cite{Genevois25}. Its completion by cubes is therefore a CAT(0) cube complex on which $J_n$ acts properly and cocompactly. Genevois used this geometry to obtain explicit solutions to the word and conjugacy problems and to study the subgroup structure of cactus groups. In particular, the pure cactus group is torsion-free, and every finite subgroup of $J_n$ is a $2$-group. These restrictions alone do not determine which finite $2$-groups embed into $J_n$ for a given $n$. Genevois proposed a description in terms of iterated wreath products in \cite[Conjecture~8.5]{Genevois25}.

In this paper, we determine which finite groups embed into cactus groups and, more generally, into generalized cactus groups associated with arbitrary Coxeter systems. Motivated by this conjecture, we obtain the following classification. The conjecture is false; we discuss its failure and give a corrected criterion in Section~\ref{sec:wreath}.

Throughout this paper, for each $n\geq1$, let $P_n\leq S_n$ be a fixed Sylow $2$-subgroup. Define
\[
 \mathcal W_0=\{1\},
 \qquad
 \mathcal W_{a+1}
 =(\mathcal W_a\times\mathcal W_a)
   \rtimes\mathbb Z/2\mathbb Z
 \quad(a\geq0),
\]
where the nontrivial element of $\mathbb Z/2\mathbb Z$ exchanges the two factors.

\begin{maintheoremA}\label{thm:mainA}
Let $n\geq2$ and let $F$ be a finite group. Write
\[
 n=\sum_{j=1}^k2^{a_j},
 \qquad a_1>\cdots>a_k\geq0,
\]
for the binary expansion of $n$. The following are equivalent:
\begin{enumerate}
 \item $F$ embeds into $J_n$;
 \item $F$ embeds into $P_n$;
 \item $F$ embeds into        $\prod_{j=1}^k\mathcal W_{a_j}$.
\end{enumerate}
\end{maintheoremA}

\medskip
Theorem~\ref{thm:mainA} was obtained independently by Marashdeh \cite[Theorem~1.1(i)]{Marashdeh26}. We extend this classification to generalized cactus groups associated with arbitrary Coxeter systems.

Let $(W,S)$ be a Coxeter system with Coxeter matrix $(m_{st})_{s,t\in S}$. For $I\subseteq S$, write $W_I=\langle I\rangle$, and let
\[
 \mathcal{I}(S)=
 \{I\subseteq S:
   I\neq\emptyset,\ W_I\text{ is finite, and }
   I\text{ is connected}\}.
\]
For $I\in\mathcal{I}(S)$, let $\omega_I$ denote the longest element of $W_I$. We introduce the terminology for Coxeter systems and specify our conventions for Coxeter diagrams in Subsection~\ref{subsec:basic}.

The \emph{generalized cactus group} associated with $(W,S)$ is defined by
\begin{equation}\label{eq:presentation}
 C_W=
 \left\langle
  \sigma_I\ (I\in\mathcal{I}(S))
  \ \middle|\
  \begin{array}{ll}
   \sigma_I^2=1,
    &I\in\mathcal{I}(S),\\
   {}[\sigma_I,\sigma_J]=1,
    &m_{st}=2
     \text{ for all }s\in I,\ t\in J,\\
   \sigma_J\sigma_I\sigma_J
    =\sigma_{\omega_J I\omega_J^{-1}},
    &I\subsetneq J
  \end{array}
 \right\rangle;
\end{equation}
see \cite{Yu23,BGPTrickle}. There is a canonical epimorphism
\[
 g_W:C_W\longrightarrow W,
 \qquad
 \sigma_I\longmapsto\omega_I.
\]
Its kernel $PC_W:=\ker g_W$ is called the \emph{generalized pure cactus group}.

For Coxeter systems with finitely many generators, Davis, Januszkiewicz, and Scott realized generalized cactus groups as the groups of all lifts of Coxeter-group actions on minimal blow-ups of reflection tilings \cite[Section~4.7]{DJS03}. Their presentations are obtained from the edges and two-dimensional cells of the universal covers. Scott studied right-angled mock reflection groups, a class that includes these generalized cactus groups. These groups are defined by actions on CAT(0) cube complexes that are simply transitive on vertices and have setwise edge stabilizers of order two. Scott described their defining relations using the squares of these complexes \cite[Section~2]{Scott08}.

Generalized cactus groups also occur in the study of Kazhdan--Lusztig cells. For finite Weyl groups, Losev constructed an action of $C_W\times C_W$ on $W$. The first factor preserves each right Kazhdan--Lusztig cell and permutes the left cells, while the second preserves each left cell and permutes the right cells \cite[Theorem~1.1]{Losev19}. Bonnaf\'e extended this construction to Coxeter systems with finitely many generators \cite[Section~5]{Bonnafe16}.

The algebraic structure of these groups has also been studied through embeddings and more general classes of group presentations. For finite Coxeter groups, Yu proved that $PC_W$ embeds into a right-angled Coxeter group and that $C_W$ is linear \cite{Yu23}. Bellingeri, Godelle, and Paris realized generalized cactus groups as trickle groups \cite{BGPTrickle}, while Genevois realized them as quandle products of copies of $\mathbb Z/2\mathbb Z$ \cite{GenevoisQuandle}. The latter description implies that their canonical Cayley graphs are median and hence gives canonical actions on CAT(0) cube complexes.

For a Coxeter diagram $X$, let $W(X)$ denote the Coxeter group defined by $X$. We write $C_{W(X)}$ for the associated generalized cactus group. In type $A_{n-1}$, identify $W(A_{n-1})$ with $S_n$ using the generating set
\[
 S=\{s_1,\ldots,s_{n-1}\},
 \qquad s_i=(i\ i+1).
\]
The correspondence
\[
 \sigma_{\{s_p,\ldots,s_{q-1}\}}
 \longleftrightarrow s_{p,q}
\]
identifies $C_{W(A_{n-1})}$ with $J_n$ and $g_W$ with $\Sigma_n$.

In joint work with Jan Kim \cite{KimKimRank}, we use cubical geometry to determine the maximal ranks of free abelian subgroups of generalized cactus groups associated with finite Coxeter systems. The finite subgroup problem is also closely related to questions about these cubical actions.

Further motivation comes from quandle products. Genevois asks for the algebraic structure of prism stabilizers and for a description of finite subgroups in quandle products \cite[Question~9.8 and Problem~9.9]{GenevoisQuandle}. Here, a prism is a subgraph that decomposes as a Cartesian product of maximal complete subgraphs. For generalized cactus groups, the prisms in the canonical Cayley graph are precisely the one-skeleta of cubes in the corresponding CAT(0) cube complex \cite[Theorem~4.12, Lemma~4.13, and Example~8.12]{GenevoisQuandle}. Thus, prism stabilizers coincide with cube stabilizers for these groups. Our proof gives a criterion for these stabilizers and uses it to classify the finite groups that embed into $C_W$.

To state the classification for generalized cactus groups, we introduce some further notation. The group $\mathcal W_a$ defined above is isomorphic to a Sylow $2$-subgroup of $S_{2^a}$ (Lemma~\ref{lem:sylow}). In particular,
\[
 \mathcal W_1\cong\mathbb Z/2\mathbb Z,
 \qquad
 \mathcal W_2\cong D_8,
\]
where $D_8$ denotes the dihedral group of order eight. For $m\geq0$, define
\begin{equation}\label{eq:Bfamily}
 \mathscr B_m=
 \left\{
  \prod_{j=1}^k
  (\mathbb Z/2\mathbb Z\times\mathcal W_{a_j}):
  k\geq0,\ a_j\geq0,\
  \sum_{j=1}^k2^{a_j}=m
 \right\}.
\end{equation}
Repetitions among the $a_j$ are allowed. An empty product is the trivial group, so $\mathscr B_0=\{1\}$.

For each finite irreducible Coxeter type $X$, let $\mathscr{L}(X)$ be the list in Table~\ref{tab:classification}. If $\mathscr{H}$ is a list of groups, write
\[
 G\times\mathscr{H}
 =\{G\times H:H\in\mathscr{H}\}.
\]

\begin{table}[htbp]
\centering\small
\renewcommand{\arraystretch}{1.3}
\begin{tabular}{@{}ll@{}}
\toprule
Irreducible type $X$ & List $\mathscr{L}(X)$\\
\midrule
$A_n$ ($n\geq1$)
 & $\{P_{n+1}\}$\\
$B_n$ ($n\geq2$)
 & $\mathscr{B}_n$\\
$D_n$ ($n\geq4$)
 & $\mathscr{B}_{2\lfloor n/2\rfloor}
    \ \cup\
    (\mathcal{W}_2\times\mathscr{B}_{n-3})$\\
$E_6$
 & $\{\mathcal{W}_2
       \times(\mathbb{Z}/2\mathbb{Z})^2\}$\\
$E_7$
 & $\{(\mathbb{Z}/2\mathbb{Z})^7,\
       \mathcal{W}_2
       \times(\mathbb{Z}/2\mathbb{Z})^4\}$\\
$E_8$
 & $\{(\mathbb{Z}/2\mathbb{Z})^8,\
       \mathcal{W}_2
       \times(\mathbb{Z}/2\mathbb{Z})^5,\
       \mathcal{W}_2^2
       \times(\mathbb{Z}/2\mathbb{Z})^2,\
       \mathcal{W}_3
       \times\mathbb{Z}/2\mathbb{Z}\}$\\
$F_4$
 & $\{(\mathbb{Z}/2\mathbb{Z})^4\}$\\
$H_3$
 & $\{(\mathbb{Z}/2\mathbb{Z})^3\}$\\
$H_4$
 & $\{(\mathbb{Z}/2\mathbb{Z})^4,\
       \mathcal{W}_2
       \times\mathbb{Z}/2\mathbb{Z}\}$\\
$I_2(p)$, $p\geq5$ odd
 & $\{\mathbb{Z}/2\mathbb{Z}\}$\\
$I_2(p)$, $p\geq6$ even
 & $\{(\mathbb{Z}/2\mathbb{Z})^2\}$\\
\bottomrule
\end{tabular}
\par\vspace{6pt}
\caption{Groups used in the classification for
finite irreducible Coxeter systems.}
\label{tab:classification}
\end{table}

For a subset $T\subseteq S$ such that $W_T$ is finite, let $X_1(T),\ldots,X_{r(T)}(T)$ be the types of the irreducible components of $(W_T,T)$, and define
\[
 \mathscr{L}(T)=
 \left\{
  H_1\times\cdots\times H_{r(T)}:
  H_i\in\mathscr{L}(X_i(T))
 \right\},
 \qquad
 \mathscr{L}(\emptyset)=\{1\}.
\]

\begin{maintheoremB}\label{thm:mainB}
Let $(W,S)$ be any Coxeter system and let $F$ be a finite group. The following are equivalent:
\begin{enumerate}
 \item $F$ embeds into $C_W$;
 \item $F$ embeds into $C_{W_T}$ for some        $T\subseteq S$ such that $W_T$ is finite;
 \item $F$ embeds into some $H\in\mathscr{L}(T)$        for some $T\subseteq S$ such that $W_T$        is finite.
\end{enumerate}
\end{maintheoremB}

Neither $W$ nor $S$ is required to be finite. If $W$ is finite, one may take $T=S$. In this case, a finite group embeds into $C_W$ if and only if it embeds into a product of one group from the corresponding row of Table~\ref{tab:classification} for each irreducible component of $(W,S)$. When $S$ is finite, we also determine the largest order and all possible orders of finite subgroups; see Corollary~\ref{cor:maximum}.

The finite subgroup structure of $C_W$ depends on the Coxeter type. Every finite subgroup of $C_W$ is a $2$-group and embeds into $W$, but these necessary conditions do not determine which finite groups embed into $C_W$. $W(B_2)$ is already a group of order eight, whereas every finite subgroup of $C_{W(B_2)}$ has order at most four. The description in type $A$ using a single finite group also fails to extend to all Coxeter types. In types $E_7$ and $E_8$, there is no finite group $G$ such that the finite groups that embed into $C_W$ are exactly those that embed into $G$.

The proof begins with the action of $C_W$ on the CAT(0) cube complex obtained as the cube-completion of its canonical Cayley graph. The Bruhat--Tits fixed point theorem implies that every finite subgroup stabilizes a cube. When $S$ is infinite, a finite subgroup uses only finitely many generators and hence lies in $C_{W_U}$ for some finite $U\subseteq S$, so the theorem applies to the corresponding complete CAT(0) cube complex. After translating the stabilized cube to contain the identity, its edges at the identity correspond to finitely many connected subsets $I\subseteq S$ with $W_I$ finite. Any two of these subsets are contained in one another or are disjoint with every generator in one commuting with every generator in the other. The maximal members are pairwise disjoint, and generators belonging to distinct maximal members commute. Their union $T$ therefore satisfies $W_T$ finite. Every vertex of the cube is a product of generators $\sigma_I$ with $I\subseteq T$, and therefore belongs to $C_{W_T}$. Since an element stabilizing the cube sends the identity to one of these vertices, the entire stabilizer is contained in $C_{W_T}$. This gives a conjugate of the original finite subgroup contained in $C_{W_T}$.

We also prove that $PC_W$ is torsion-free (Lemma~\ref{lem:pure}). Consequently, the restriction of $g_W$ to every finite subgroup of $C_W$ is injective. We can therefore study a cube stabilizer as a subgroup of $W$. By Theorem~\ref{thm:criterion}, a vertex $v$ stabilizes the cube exactly when conjugation by $g_W(v)$ preserves the collection of subgroups $W_I$ corresponding to its edges at the identity. In type $A$, we embed finite subgroups of $J_n$ into $P_n$ and explicitly realize $P_n$ in $J_n$. This requires no analysis of cube stabilizers. For type $I_2(p)$, we use the defining presentation directly. For the other types, we analyze cube stabilizers to obtain the groups in Table~\ref{tab:classification}. We also construct every group in that table as a subgroup of the corresponding generalized cactus group. Products of Coxeter systems and the preceding reduction to finite Coxeter groups complete the proof.

A manuscript containing our results for ordinary cactus groups was communicated to Genevois in July 2026. Marashdeh's manuscript was completed in early September 2026. Our extension to generalized cactus groups was developed after the July version. Marashdeh’s paper also computes the cube stabilizers explicitly in terms of laminar families of intervals and determines which of them are conjugate to subgroups generated by generators, studies maximal finite subgroups, and identifies conjugacy classes of finite subgroups with connected components of fixed loci on the quotient cube complex.

For ordinary cactus groups, we also give an exact criterion using products of iterated wreath products with a recursively defined weight; see Section~\ref{sec:wreath}. We give a corresponding criterion in type $B$. Marashdeh independently disproved Genevois's conjecture at $n=10$ when all parameters are required to be at least two \cite[Corollary~1.4(b)]{Marashdeh26}. Our example at $n=20$ shows that $P_{20}$ is not isomorphic to any single iterated wreath product of the proposed form, even when parameters equal to one are allowed. Thus, the conjecture cannot be corrected merely by changing its numerical bound.

The paper is organized as follows. In Section~\ref{sec:common}, we describe cubes and their stabilizers, prove the reduction to spherical subsets, and discuss products of Coxeter systems. Section~\ref{sec:typeA} proves Theorem~\ref{thm:mainA}. Sections~\ref{sec:typeB} and~\ref{sec:typeD} classify the finite groups that embed into generalized cactus groups of types $B$ and $D$, respectively. In Section~\ref{sec:typeE}, we analyze cube stabilizers and construct the required finite subgroups in types $E_6$, $E_7$, and $E_8$. Section~\ref{sec:small} covers types $F_4$, $H_3$, $H_4$, and $I_2(p)$. Section~\ref{sec:consequences} proves Theorem~\ref{thm:mainB} and determines all possible orders of finite subgroups when $S$ is finite. Finally, Section~\ref{sec:wreath} gives weighted embedding criteria for types $A$ and $B$ and concludes with counterexamples to Genevois's conjecture.

\vspace{0.6cm} 

\noindent \textbf{Acknowledgment.} 
The author thanks Anthony Genevois for reading an earlier draft on finite subgroups of ordinary cactus groups. The author also thanks Mohammad F. Marashdeh for helpful discussions and for his encouragement during the preparation of this paper. The author was supported by Israel Science Foundation Grant 3423/24 (PI Ilya Gekhtman).

\vspace{0.6cm} 

\noindent \textbf{AI disclosure.} 
The authors used ChatGPT (OpenAI) to assist with language editing, improve the presentation of the manuscript, and review mathematical arguments. All AI-assisted content was reviewed, edited, and verified by the authors, who take full responsibility for the integrity of the work.

\section{Preliminaries and cube stabilizers}\label{sec:common}
\subsection{Coxeter groups}
\label{subsec:basic}

Throughout this section, $(W,S)$ is an arbitrary Coxeter system unless stated otherwise. We use the notation for generalized cactus groups introduced in Section~\ref{sec:introduction}.

Recall that a \emph{Coxeter system} is a group $W$ with a generating set $S$ and a presentation
\[
 W=
 \left\langle
 S\ \middle|\
 (st)^{m_{st}}=1
 \text{ whenever }m_{st}<\infty
 \right\rangle,
\]
where $m_{ss}=1$ for every $s\in S$, and $m_{st}=m_{ts}\in\{2,3,\ldots,\infty\}$ for distinct $s,t\in S$. The Coxeter diagram has vertex set $S$, with an edge between distinct $s,t$ whenever $m_{st}\neq2$. We label each edge by $m_{st}$, omitting the label when $m_{st}=3$. Connectedness of a subset of $S$ refers to its induced subdiagram.

The \emph{length} $\ell(w)$ of $w\in W$ is the minimum length of a word in $S$ representing $w$. An expression of this length is called a \emph{reduced expression}. The \emph{rank} of $(W,S)$ is $|S|$. The system is called \emph{finite} if $W$ is finite and \emph{irreducible} if its Coxeter diagram is connected.

For $T\subseteq S$, the subgroup $W_T$ is called a \emph{standard parabolic subgroup}. The pair $(W_T,T)$ is itself a Coxeter system. A subset $T\subseteq S$ is called \emph{spherical} if $W_T$ is finite. Two subsets are called \emph{nested} if one is contained in the other. For $I,J\in\mathcal{I}(S)$, write $I\perp J$ if $I$ and $J$ are disjoint and every generator in $I$ commutes with every generator in $J$. We then say that $I$ and $J$ are \emph{orthogonal}. Equivalently, $m_{st}=2$ for every $s\in I$ and $t\in J$.

For $J\in\mathcal{I}(S)$, conjugation by the longest element $\omega_J$ preserves $J$ and induces an involutive automorphism of its Coxeter diagram, called \emph{opposition}. For $I\subseteq J$, write
\[
 \omega_J(I)
 =\omega_J I\omega_J^{-1}
 =\{\omega_J s\omega_J^{-1}:s\in I\}.
\]
If $I\in\mathcal{I}(S)$, then $\omega_J(I)\in\mathcal{I}(S)$. We refer to \cite{BjornerBrenti} for background on Coxeter groups.

We use the classification of finite irreducible Coxeter systems given in \cite[Appendix~A1, Table~I]{BjornerBrenti}. Their types are
\[
 \begin{gathered}
  A_n\ (n\geq1),\qquad
  B_n\ (n\geq2),\qquad
  D_n\ (n\geq4),\\
  E_6,\ E_7,\ E_8,\ F_4,\ H_3,\ H_4,\qquad
  I_2(p)\ (p\geq5).
 \end{gathered}
\]
Every finite Coxeter system is a finite product of irreducible Coxeter systems of these types.

\subsection{Cube complexes}\label{subsec:basic2}

A \emph{cube complex} is obtained by gluing Euclidean unit cubes $[0,1]^d$ along faces by isometries. We equip it with the induced length metric. Write $V(X)$ for the vertex set of a cube complex $X$ and $X^{(1)}$ for its graph of vertices and edges, called its \emph{one-skeleton}.

For a vertex $v$ of a cube complex, its \emph{link} is the simplicial complex whose vertices correspond to the edges incident to $v$. A collection of these vertices spans a simplex precisely when the corresponding edges lie in a common cube. A simplicial complex is \emph{flag} if every finite collection of pairwise adjacent vertices spans a simplex.

A cube complex is called \emph{non-positively curved} if the link of every vertex is flag. It is called \emph{CAT(0)} if it is also simply connected. This characterization agrees with the CAT(0) condition for the induced length metric; see \cite[Theorem~B.8]{Leary13}. In particular, a finite collection of edges incident to a vertex in a CAT(0) cube complex lies in a common cube whenever every pair lies in a common square.

A connected graph $X$ is \emph{median} if, for any $x_1,x_2,x_3\in V(X)$, there exists a unique vertex $m\in V(X)$ such that
\[
 d(x_i,x_j)=d(x_i,m)+d(m,x_j)
 \qquad(1\leq i<j\leq3),
\]
where $d$ is the graph metric. The \emph{cube-completion} of $X$ is the cube complex obtained by filling each subgraph isomorphic to the one-skeleton of a cube with a cube. A connected graph is median if and only if its cube-completion is a CAT(0) cube complex \cite{Chepoi00}.

Let $X$ be a CAT(0) cube complex. For two edges $e,f$ of $X$, write $e\sim f$ if there exists a finite sequence $e=e_0,\ldots,e_k=f$, with $k\geq0$, such that for each $0\leq i<k$, the edges $e_i$ and $e_{i+1}$ lie in the boundary of a common square and have no common endpoint. A \emph{midcube} of a cube identified with $[0,1]^d$ is a subset of the form
\[
 \{(t_1,\ldots,t_d)\in[0,1]^d:t_j=1/2\},
 \qquad 1\leq j\leq d.
\]
For an equivalence class $\mathcal E$ of edges, let $H_{\mathcal E}\subseteq X$ be the union of all midcubes of cubes of $X$ that contain the midpoint of an edge in $\mathcal E$. We call $H_{\mathcal E}$ the \emph{hyperplane} associated with $\mathcal E$. The edges in $\mathcal E$ are said to be \emph{dual} to $H_{\mathcal E}$. For every hyperplane $H$, the complement $X\setminus H$ has exactly two connected components; see \cite[Section~6.2]{Chepoi00}. We say that $H$ \emph{separates} two vertices if they lie in different components of $X\setminus H$.

All group actions on cube complexes considered below are by \emph{cubical automorphisms}, which send each cube isometrically onto a cube. For a group $G$ acting on $X$ and a cube $Q\subseteq X$, write
\[
 \operatorname{Stab}_G(Q)
 =\{g\in G:gQ=Q\}.
\]
An element of this subgroup may permute the vertices of $Q$.

We will use the following consequence of the Bruhat--Tits fixed point theorem. If a finite group acts by cubical automorphisms on a complete CAT(0) cube complex, then it stabilizes a cube. Indeed, the fixed point theorem gives a fixed point \cite[Corollary~II.2.8]{BridsonHaefliger99}. The unique cube whose relative interior contains that point is then preserved by the group.

\subsection{Compatible families and cubes}
The subgroup $PC_W$ has finite index in $C_W$ exactly when $W$ is finite. A finite family $\mathcal{A}\subseteq\mathcal{I}(S)$ is called \emph{compatible} if every pair of distinct members is nested or orthogonal.

The Cayley graph $\mathscr{C}_W$ of $C_W$ with respect to the generating set $\{\sigma_{I} : I\in\mathcal{I}(S)\}$ is median, as follows from the quandle-product description of $C_W$ and \cite[Theorem~4.12 and Lemma~4.13]{GenevoisQuandle}. Its cube-completion, denoted by $\mathscr{A}_W$, is therefore a CAT(0) cube complex. Left multiplication induces a cubical action of $C_W$ on $\mathscr{A}_W$ that is free and transitive on vertices. If $S$ is finite, then $\mathcal{I}(S)$ is finite, so $\mathscr{A}_W$ is locally finite and finite-dimensional.

\begin{lemma}\label{lem:compatible-cubes}
Cubes of $\mathscr{A}_W$ containing the identity are in bijection with compatible families $\mathcal{A}$. Any square in $\mathscr{A}_W$ is a left translate of a defining square associated to a commuting relation or a nesting relation in \eqref{eq:presentation}.
\end{lemma}
\begin{proof}
The description of squares is contained in \cite[Lemma~4.15]{GenevoisQuandle}. For distinct $I,J\in\mathcal{I}(S)$, the edges $(1,\sigma_I)$ and $(1,\sigma_J)$ lie in a common square if and only if $I$ and $J$ are nested or orthogonal. Since the link at the identity is flag, a finite family $\mathcal A\subseteq\mathcal{I}(S)$ determines a unique cube whose edges incident to the identity are precisely $(1,\sigma_I)$ for $I\in\mathcal{A}$ if and only if every pair of distinct members of $\mathcal{A}$ is nested or orthogonal. This is precisely the compatibility condition.
\end{proof}

For each compatible family $\mathcal{A}$, we denote the corresponding cube containing the identity by $Q_{\mathcal{A}}$.

\begin{proposition}\label{prop:parabolic}
For every $T\subseteq S$, the natural map $C_{W_T}\to C_W$ is injective. Identifying $C_{W_T}$ with its image, its canonical Cayley graph is the subgraph of $\mathscr{C}_W$ with vertex set $C_{W_T}$ and edges $(h,h\sigma_I)$, where $h\in C_{W_T}$ and $I\in\mathcal{I}(S)$ satisfies $I\subseteq T$. This inclusion extends to a cubical embedding $\mathscr{A}_{W_T}\hookrightarrow\mathscr{A}_W$.
\end{proposition}
\begin{proof}
Let $\mathcal{I}(T)=\{I\in\mathcal{I}(S):I\subseteq T\}$. For $I,J\in\mathcal{I}(T)$ with $I\subsetneq J$, conjugation by $\omega_J$ is a diagram automorphism of $J$. Hence, $\omega_JI\omega_J^{-1}$ is a connected spherical subset of $J\subseteq T$, and therefore belongs to $\mathcal{I}(T)$. This implies that $\mathcal{I}(T)$ is closed under the opposition maps appearing in the nesting relations, which means that it is stable in the sense of \cite[Definition~2.14]{GenevoisQuandle}.

By \cite[Corollary~3.7]{GenevoisQuandle}, the subgroup generated by $\sigma_I$ with $I\in\mathcal{I}(T)$ has a presentation consisting of these generators and the corresponding relations in \eqref{eq:presentation}. This is precisely the presentation of $C_{W_T}$, proving the injectivity of the natural map. Under this identification, the image of its canonical Cayley graph has vertex set $C_{W_T}$ and edges $(h,h\sigma_I)$ with $h\in C_{W_T}$ and $I\in\mathcal{I}(T)$, as asserted.

To obtain the cubical embedding, consider a cube containing the identity in $\mathscr{A}_{W_T}$, whose edges at the identity are $(1,\sigma_I)$ for $I$ in a compatible family $\mathcal{A}$. Inclusion and orthogonality among members of $\mathcal{I}(T)$ are unchanged in $\mathcal{I}(S)$. Thus, Lemma~\ref{lem:compatible-cubes} gives a unique cube in $\mathscr{A}_W$ with the same edges at the identity. Under the identification $C_{W_T}\subseteq C_W$, the two cubes have the same vertices and edges. Translating by elements of $C_{W_T}$ identifies every cube of $\mathscr{A}_{W_T}$ with a cube of $\mathscr{A}_W$. These identifications preserve faces and hence identify $\mathscr{A}_{W_T}$ with a cubical subcomplex of $\mathscr{A}_W$.
\end{proof}

We will write $C_T\leq C_W$ for the image of the group $C_{W_T}$ embedded in $C_W$, even when $T$ is not connected or spherical.

The following elementary observation gives the dimension bound needed later.

\begin{lemma}\label{lem:dimension}
Every compatible family $\mathcal{A}$ satisfies $|\mathcal{A}|\leq|\bigcup_{I\in\mathcal{A}}I|$. In particular, $|\mathcal{A}|\leq|S|$ when $S$ is finite. Moreover, if $M\in\mathcal{A}$, the union of all members of $\mathcal{A}$ properly contained in $M$ is a proper subset of $M$.
\end{lemma}
\begin{proof}
We first prove the last assertion. For each $M\in\mathcal{A}$, let
\[
 U_M=\bigcup_{\substack{I\in\mathcal{A}\\ I\subsetneq M}} I.
\]
We claim that $U_M\subsetneq M$. If there are no members of $\mathcal{A}$ properly contained in $M$, this is immediate. Otherwise, let $M_1,\ldots,M_r$ be the maximal members among those properly contained in $M$. Every member occurring in the definition of $U_M$ is contained in some $M_i$, so
\[
 U_M=M_1\cup\cdots\cup M_r.
\]
By maximality, no $M_i$ is contained in another $M_j$, so compatibility implies that the $M_i$ are pairwise orthogonal. If $r=1$, then $U_M=M_1\subsetneq M$. If $r\geq2$, the induced subdiagram on $U_M$ is disconnected, whereas the subdiagram on $M$ is connected. Thus, $U_M\neq M$ in either case.

For each $M\in\mathcal{A}$, choose $t_M\in M\setminus U_M$. We show that these vertices are distinct. If neither $I\subseteq J$ nor $J\subseteq I$, then compatibility implies that $I$ and $J$ are orthogonal. Hence, $I\cap J=\emptyset$, so $t_I\neq t_J$. If $I\subsetneq J$, then $I\subseteq U_J$, whence $t_J\notin I$, while $t_I\in I$. The remaining case is symmetric. Consequently,
\[
 \mathcal{A}\longrightarrow\bigcup_{I\in\mathcal{A}}I,
 \qquad M\longmapsto t_M,
\]
is injective, proving the required inequalities.
\end{proof}

For $\mathcal{U}\subseteq\mathcal{A}$, we order its members as $I_1,\ldots,I_k$ so that $I_i\subsetneq I_j$ implies $i<j$, and put
\begin{equation}\label{eq:vertices}
 v_{\mathcal{U}}=\sigma_{I_1}\cdots\sigma_{I_k},\qquad w_{\mathcal{U}}=g_W(v_{\mathcal{U}})=\omega_{I_1}\cdots\omega_{I_k}.
\end{equation}

\begin{lemma}\label{lem:vertices}
The expressions in \eqref{eq:vertices} are independent of the chosen ordering satisfying the above condition, and $V(Q_{\mathcal{A}}) =\{v_{\mathcal{U}}:\mathcal{U}\subseteq\mathcal{A}\}$. The map $\mathcal{U}\mapsto v_{\mathcal{U}}$ is a bijection from the set of subsets of $\mathcal{A}$ to $V(Q_{\mathcal{A}})$. For each $\mathcal{U}\subseteq\mathcal{A}$, the hyperplanes separating the identity and $v_{\mathcal{U}}$ are precisely those dual to the edges $(1,\sigma_I)$ with $I\in\mathcal{U}$.
\end{lemma}
\begin{proof}
Any two orderings satisfying the above condition are related by a sequence of interchanges of adjacent members $I,J$ such that neither $I\subseteq J$ nor $J\subseteq I$. By compatibility, such members are orthogonal, so their generators, as well as their longest elements, commute. Thus, the products defining $v_{\mathcal{U}}$ and $w_{\mathcal{U}}$ are independent of the chosen ordering.

For $\mathcal{U}\subseteq\mathcal{A}$, the cube $Q_{\mathcal{U}}$ is the face of $Q_{\mathcal{A}}$ whose edges incident to the identity are precisely $(1,\sigma_I)$ for $I\in\mathcal{U}$. We show by induction on $|\mathcal{U}|$ that $v_{\mathcal{U}}$ is the vertex of $Q_{\mathcal{U}}$ opposite the identity. For $\mathcal{U}=\emptyset$, this holds since $Q_{\mathcal{U}}=\{1\}$ and $v_{\mathcal{U}}=1$.

Suppose $\mathcal{U}\neq\emptyset$, and choose $I\in\mathcal{U}$ such that no member of $\mathcal{U}$ is properly contained in $I$. For every $J\in\mathcal{U}\setminus\{I\}$, either $I\subsetneq J$ or $I\perp J$, so the corresponding defining square has vertices, in cyclic order,
\[
 1,\quad \sigma_{I},\quad \sigma_{I}\sigma_{J},\quad \sigma_{J}.
\]
These squares show that the face of $Q_{\mathcal{U}}$ opposite $Q_{\mathcal{U}\setminus\{I\}}$ has precisely the edges
\[
 (\sigma_{I},\sigma_{I}\sigma_{J}),
 \qquad J\in\mathcal{U}\setminus\{I\},
\]
incident to $\sigma_{I}$. Since a cube is uniquely determined by its edges at a vertex, this face is $\sigma_{I} Q_{\mathcal{U}\setminus\{I\}}$. By induction, its vertex opposite $\sigma_{I}$ is $\sigma_{I} v_{\mathcal{U}\setminus\{I\}}$. We may place $I$ first in an ordering satisfying the above condition, so this vertex is $v_{\mathcal{U}}$. It is also the vertex of $Q_{\mathcal{U}}$ opposite the identity.

Every vertex of $Q_{\mathcal{A}}$ is opposite the identity in a unique face of $Q_{\mathcal{A}}$ containing the identity. Applying this to the faces $Q_{\mathcal{U}}$ proves the asserted description of $V(Q_{\mathcal{A}})$ and the bijectivity of $\mathcal{U}\mapsto v_{\mathcal{U}}$. Moreover, the hyperplanes separating the identity from $v_{\mathcal{U}}$ are exactly those crossing $Q_{\mathcal{U}}$, namely those dual to the edges $(1,\sigma_I)$ with $I\in\mathcal{U}$.
\end{proof}

\begin{lemma}\label{lem:pure}
The action of $PC_W$ on $\mathscr{A}_W$ has trivial cube stabilizers, and $PC_W$ is torsion-free. Every finite subgroup of $C_W$ is a $2$-group, and the restriction of $g_W$ to each such subgroup is injective.
\end{lemma}
\begin{proof}
Let $\mathcal{A}$ be a compatible family. We first show that $w_{\mathcal{U}}\ne1$ for every nonempty $\mathcal{U}\subseteq\mathcal{A}$. Let $M_1,\ldots,M_r$ be the maximal members of $\mathcal{U}$. These sets are pairwise orthogonal. The longest elements indexed by subsets of different $M_i$ commute. Grouping together the factors $\omega_I$ with $I\subseteq M_i$, we obtain $w_{\mathcal{U}}=w_1\cdots w_r$ for some $w_i\in W_{M_i}$.

Since the $M_i$ are pairwise orthogonal, $w_{\mathcal{U}}=1$ if and only if $w_i=1$ for every $i$. For a maximal member $M$, the corresponding factor in $w_{\mathcal{U}}$ has the form $u \omega_M$, where $u\in W_N$ and $N$ is the union of the proper members of $\mathcal{U}$ contained in $M$. By Lemma~\ref{lem:dimension}, $N\subsetneq M$. Every reduced expression for $\omega_M$ contains all elements of $M$, whereas every element of $W_N$ admits a reduced expression using only elements of $N$; see \cite[Corollaries~1.4.6(ii) and~1.4.8(i)--(ii)] {BjornerBrenti}. This implies that $\omega_M\notin W_N$, and so $u\omega_M\ne1$, proving that $w_{\mathcal{U}}\ne1$. Together with Lemma~\ref{lem:vertices}, this gives $V(Q_{\mathcal{A}})\cap PC_W=\{1\}$ for every compatible family $\mathcal{A}$.

Now let $p\in PC_W$ stabilize a cube $Q$, and choose a vertex $a\in V(Q)$. The translated cube $a^{-1}Q$ contains the identity, so it equals $Q_{\mathcal{A}}$ for some compatible family $\mathcal{A}$. Since $pa\in V(Q)$ and $PC_W$ is normal, we have $a^{-1}pa\in PC_W\cap V(a^{-1}Q)=\{1\}$. Hence $p=1$, proving that the action of $PC_W$ on $\mathscr{A}_W$ has trivial cube stabilizers.

Let $F\leq C_W$ be finite. Choose words in the generators $\sigma_I$ representing a finite generating set of $F$, and let $U\subseteq S$ be the union of all subsets $I$ for which $\sigma_I$ occurs in these words. Each such $I$ is spherical and hence finite, so $U$ is finite. By Proposition~\ref{prop:parabolic}, we have $F\leq C_U$.

Since $U$ is finite, the CAT(0) cube complex $\mathscr{A}_{W_U}$ is locally finite and finite-dimensional, and hence complete. By the fixed point argument in Subsection~\ref{subsec:basic2}, $F$ stabilizes a cube of $\mathscr{A}_{W_U}$. By Proposition~\ref{prop:parabolic}, this is also a cube of $\mathscr{A}_W$. Since the action of $PC_W$ on $\mathscr{A}_W$ has trivial cube stabilizers, we have $F\cap PC_W=\{1\}$. Hence, the restriction of $g_W$ to $F$ is injective. If this cube has dimension $d$, then $F$ acts freely on its $2^d$ vertices, since $C_W$ acts freely on vertices. Thus, $|F|$ divides $2^d$, proving that $F$ is a $2$-group. Finally, every finite subgroup $H\leq PC_W$ is also a finite subgroup of $C_W$, so the preceding argument gives $H=H\cap PC_W=\{1\}$. Therefore, $PC_W$ is torsion-free.
\end{proof}

\subsection{Reduction to finite Coxeter groups}
The proof of Lemma~\ref{lem:pure} places every finite subgroup of $C_W$ inside $C_U$ for some finite subset $U\subseteq S$, although $W_U$ need not be finite. We now show that, after conjugation, the subset can be chosen to be spherical. This reduces the classification of finite subgroups of $C_W$ to the case of finite Coxeter groups.

\begin{theorem}\label{thm:localization}
Let $(W,S)$ be any Coxeter system. For every finite subgroup $F\leq C_W$, there exist a spherical subset $T\subseteq S$ and an element $a\in C_W$ such that $a^{-1}Fa\leq C_T$.

For a compatible family $\mathcal{A}$, let $M_1,\ldots,M_r$ be its maximal members and set $T=\bigcup_{i=1}^r M_i$. Then $T$ is spherical and
\begin{equation}\label{eq:local-stabilizer}
 \operatorname{Stab}_{C_W}(Q_{\mathcal{A}})
 =\operatorname{Stab}_{C_T}(Q_{\mathcal{A}}).
\end{equation}
\end{theorem}
\begin{proof}
If $\mathcal{A}=\emptyset$, then $T=\emptyset$ is spherical and $Q_{\mathcal{A}}=\{1\}$. Both stabilizers in \eqref{eq:local-stabilizer} are trivial. We may therefore assume that $\mathcal{A}\neq\emptyset$. To prove \eqref{eq:local-stabilizer}, it suffices to show that $\operatorname{Stab}_{C_W}(Q_{\mathcal{A}}) \subseteq \operatorname{Stab}_{C_T}(Q_{\mathcal{A}})$.

Distinct maximal members of $\mathcal{A}$ cannot be nested, so they are pairwise orthogonal. Each $M_i$ is spherical by the definition of the generating set. Consequently, $W_T\cong W_{M_1}\times\cdots\times W_{M_r}$ is finite. Every member of $\mathcal{A}$ lies in one of the $M_i$, so Lemma~\ref{lem:vertices} gives $V(Q_\mathcal{A})\subseteq C_T$. If $h\in\operatorname{Stab}_{C_W}(Q_\mathcal{A})$, then $h=h\cdot1\in V(Q_\mathcal{A})\subseteq C_T$. Thus, the stabilizer $\operatorname{Stab}_{C_W}(Q_{\mathcal{A}})$ is contained in $C_T$, proving \eqref{eq:local-stabilizer}.

Let $F\leq C_W$ be finite. As shown in the proof of Lemma~\ref{lem:pure}, $F$ stabilizes a cube $Q$ of $\mathscr{A}_W$. Choose a vertex $a\in V(Q)\subseteq C_W$. Then $a^{-1}Q=Q_{\mathcal{A}}$ for a compatible family $\mathcal{A}$, and $a^{-1}Fa$ stabilizes $Q_{\mathcal{A}}$. Let $T$ be the union of the maximal members of $\mathcal{A}$. The preceding argument shows that $T$ is spherical, and \eqref{eq:local-stabilizer} gives $a^{-1}Fa\leq C_T$.
\end{proof}

\subsection{Cube stabilizers}
Let $K_\mathcal A=\operatorname{Stab}_{C_W}(Q_\mathcal A)$ and $\mathcal P_\mathcal A=\{W_I:I\in\mathcal A\}$. The following theorem shows that a vertex $v_{\mathcal{U}}$ of $Q_{\mathcal{A}}$ belongs to $K_{\mathcal{A}}$ if and only if conjugation by $w_{\mathcal{U}}$ permutes the subgroups in $\mathcal P_{\mathcal{A}}$.

\begin{theorem}\label{thm:criterion}
For every compatible family $\mathcal A$,
\begin{equation}\label{eq:criterion}
 K_\mathcal A=\{v_{\mathcal{U}}:\mathcal{U}\subseteq\mathcal{A},\ w_{\mathcal{U}}\mathcal P_{\mathcal{A}} w_{\mathcal{U}}^{-1}=\mathcal P_{\mathcal{A}}\}.
\end{equation}
The restriction of $g_W$ identifies $K_\mathcal A$ with the finite $2$-subgroup
\[
 H_{\mathcal{A}}:=\{w_{\mathcal{U}}:\mathcal{U}\subseteq\mathcal A,\ w_{\mathcal{U}}\mathcal P_\mathcal A w_{\mathcal{U}}^{-1}=\mathcal P_\mathcal A\}\leq W.
\]
Thus, a finite group $F$ embeds into $C_W$ if and only if $F$ embeds into $H_\mathcal A$ for some compatible family $\mathcal A$.
\end{theorem}
\begin{proof}
For an oriented edge $e=(h,h\sigma_I)$, define $\lambda(e)=g_W(h)W_Ig_W(h)^{-1}$. Since $\omega_I$ normalizes $W_I$, reversing the edge does not change its label. One can check that opposite edges of a square obtained by defining commuting or nesting relation have the same label. By Lemma~\ref{lem:compatible-cubes}, the same holds for every square. Thus, $\lambda$ is constant on the edges dual to a given hyperplane, and we denote this common value by $\lambda(\mathfrak h)$. For every $x\in C_W$ and every hyperplane $\mathfrak h$, the definition gives $\lambda(x\mathfrak h)=g_W(x)\lambda(\mathfrak h)g_W(x)^{-1}$. Moreover, distinct edges incident to any vertex have distinct labels, since $W_I=W_J$ implies $I=J$.

We notice that the hyperplanes crossing $Q_{\mathcal{A}}$ have label set $\mathcal P_{\mathcal{A}}$. Let $h\in K_{\mathcal{A}}$. Since $1\in V(Q_{\mathcal{A}})$ and $hQ_{\mathcal{A}}=Q_{\mathcal{A}}$, we have $h=h\cdot1\in V(Q_{\mathcal{A}})$. By Lemma~\ref{lem:vertices}, we can write $h=v_{\mathcal{U}}$ for some $\mathcal{U}\subseteq\mathcal{A}$. Equivariance then gives $w_{\mathcal{U}}\mathcal P_{\mathcal{A}}w_{\mathcal{U}}^{-1} =\mathcal P_{\mathcal{A}}$. This proves that $K_{\mathcal{A}}$ is contained in the right-hand side of \eqref{eq:criterion}.

Conversely, let $\mathcal{U}\subseteq\mathcal{A}$ satisfy  $w_{\mathcal{U}}\mathcal P_{\mathcal{A}}w_{\mathcal{U}}^{-1}  =\mathcal P_{\mathcal{A}}$. The cubes $Q_{\mathcal{A}}$ and $v_{\mathcal{U}}Q_{\mathcal{A}}$ have the common vertex $v_{\mathcal{U}}$. By equivariance and the assumed equality, the sets of labels of the hyperplanes crossing these cubes coincide. In each cube, every hyperplane crossing it is dual to exactly one edge incident to $v_{\mathcal{U}}$. Thus, the sets of labels of these incident edges also coincide. Since distinct edges incident to $v_{\mathcal{U}}$ have distinct labels, the two cubes have exactly the same edges incident to this vertex. A cube is uniquely determined by its edges incident to a vertex, so $Q_{\mathcal{A}}=v_{\mathcal{U}}Q_{\mathcal{A}}.$ Therefore, $v_{\mathcal{U}}\in K_{\mathcal{A}}$, proving the reverse inclusion in \eqref{eq:criterion}.

The inclusion $K_{\mathcal{A}}\subseteq V(Q_{\mathcal{A}})$ shows that $K_{\mathcal{A}}$ is finite. By Lemma~\ref{lem:pure}, it is a $2$-group and the restriction of $g_W$ to $K_{\mathcal{A}}$ is injective. Applying $g_W$ to \eqref{eq:criterion} gives $g_W(K_{\mathcal{A}})=H_{\mathcal{A}}$. Thus, this restriction is an isomorphism $K_{\mathcal{A}}\cong H_{\mathcal{A}}$, and $H_{\mathcal{A}}$ is a finite $2$-subgroup of $W$.

Finally, suppose that a finite group $F$ embeds into $C_W$, and identify $F$ with its image. As shown in the proof of Lemma~\ref{lem:pure}, $F$ stabilizes a cube $Q$. Choose a vertex $a\in V(Q)$. Then $a^{-1}Q=Q_{\mathcal{B}}$ for some compatible family $\mathcal{B}$, and $a^{-1}Fa\leq K_{\mathcal{B}}\cong H_{\mathcal{B}}$. Hence, $F$ embeds into $H_{\mathcal{B}}$. Conversely, if $F$ embeds into $H_{\mathcal{A}}$ for some compatible family $\mathcal{A}$, then $H_{\mathcal{A}}\cong K_{\mathcal{A}}\leq C_W$ gives an embedding of $F$ into $C_W$.
\end{proof}

The preceding theorem reduces the problem to the study of cube stabilizers. For each finite irreducible Coxeter type, we show that every cube stabilizer embeds into one of the groups in the corresponding row of Table~\ref{tab:classification}. Conversely, we also show that every group in that row is realized as a subgroup of $C_W$.

We next decompose a cube stabilizer according to the maximal members of its compatible family. Since each maximal member is connected and spherical, the following lemma reduces the computation to cube stabilizers for finite irreducible Coxeter systems.

\begin{lemma}\label{lem:maxsupport}
Let $M_1,\ldots,M_k$ be the maximal members of a compatible family $\mathcal A$, and let $\mathcal A_i=\{I\in\mathcal A:I\subseteq M_i\}$. Then $K_\mathcal A\cong\prod_{i=1}^k\operatorname{Stab}_{C_{M_i}}(Q_{\mathcal A_i})$.
\end{lemma}
\begin{proof}
The $M_i$ are pairwise orthogonal, and every member of $\mathcal A$ lies in exactly one $M_i$. Fix $\mathcal{U}\subseteq \mathcal{A}$ and write $\mathcal{U}_i=\mathcal{U}\cap\mathcal A_i$. By Lemma~\ref{lem:vertices}, every cube vertex factors uniquely as $v_{\mathcal{U}}=\prod_i v_{\mathcal{U}_i}$, with commuting factors. Let $w_i=g_W(v_{\mathcal{U}_i})\in W_{M_i}$. For $I\in\mathcal A_i$, all $w_j$ with $j\ne i$ centralize $W_I$. Moreover, if $w_iW_Iw_i^{-1}=W_J$ for $J\in\mathcal A$, then $W_J\leq W_{M_i}$, which implies $J\subseteq M_i$ and hence $J\in\mathcal A_i$. Consequently,
\[
 w_{\mathcal{U}}\mathcal P_\mathcal A w_{\mathcal{U}}^{-1}=\mathcal P_\mathcal A
 \quad\Longleftrightarrow\quad
 w_i\mathcal P_{\mathcal A_i}w_i^{-1}=\mathcal P_{\mathcal A_i}\quad\text{for every }i.
\]
By Theorem~\ref{thm:criterion}, applied to $C_W$ and each $C_{M_i}$, we have
\[
 v_{\mathcal{U}}\in K_{\mathcal{A}}
 \quad\Longleftrightarrow\quad
 v_{\mathcal{U}_i}\in
 \operatorname{Stab}_{C_{M_i}}(Q_{\mathcal{A}_i})
 \quad\text{for every }i.
\]
For $i\ne j$, the generators $\sigma_I$ with $I\subseteq M_i$ commute with the generators $\sigma_J$ with $J\subseteq M_j$, since $I\perp J$. This implies that every element of $C_{M_i}$ commutes with every element of $C_{M_j}$. Hence, multiplication defines a homomorphism
\[
 \prod_{i=1}^k\operatorname{Stab}_{C_{M_i}}(Q_{\mathcal{A}_i})
 \longrightarrow K_{\mathcal{A}},
 \qquad
 (h_1,\ldots,h_k)\longmapsto h_1\cdots h_k.
\]
The equivalence above shows that this map is surjective. To prove injectivity, suppose that $h_1\cdots h_k=1$. Since each $h_i$ stabilizes $Q_{\mathcal{A}_i}$, which contains the identity, we have $h_i\in V(Q_{\mathcal{A}_i})$. By Lemma~\ref{lem:vertices}, write $h_i=v_{\mathcal{U}_i}$ for some $\mathcal{U}_i\subseteq\mathcal{A}_i$. Letting $\mathcal{U}=\bigcup_i\mathcal{U}_i$, we obtain
\[
 v_{\mathcal{U}}
 =\prod_i v_{\mathcal{U}_i}
 =h_1\cdots h_k
 =1=v_{\emptyset}.
\]
Lemma~\ref{lem:vertices} implies $\mathcal{U}=\emptyset$, so every $\mathcal{U}_i$ is empty and every $h_i$ is the identity. Therefore, the homomorphism is injective.
\end{proof}

If a compatible family $\mathcal A$ contains $S$, then $S$ is its unique maximal member. In this case, Lemma~\ref{lem:maxsupport} does not give a decomposition into stabilizers associated with proper subsets of $S$. When $\omega_S$ is central in $W$, the following lemma shows that $K_{\mathcal A} \cong K_{\mathcal A\setminus\{S\}}\times\mathbb Z/2\mathbb Z$. Then we can apply Lemma~\ref{lem:maxsupport} to $\mathcal A\setminus\{S\}$, whose maximal members are proper subsets of $S$.

\begin{lemma}\label{lem:central}
Let $(W,S)$ be a finite irreducible Coxeter system such that $\omega_S$ is central in $W$. Let $\mathcal{A}$ be compatible, $\mathcal D=\mathcal{A}\setminus\{S\}$, and let $\delta=1$ if $S\in\mathcal{A}$, and $0$ otherwise. Then $K_\mathcal{A}\cong K_{\mathcal D}\times(\mathbb{Z}/2\mathbb{Z})^{\delta}$.
\end{lemma}
\begin{proof}
If $S\notin\mathcal{A}$, then $\mathcal D=\mathcal{A}$ and $\delta=0$, so the isomorphism $K_\mathcal{A}\cong K_{\mathcal D}$ is immediate. Suppose that $S\in\mathcal{A}$. Since $\omega_S$ is central in $W$, $\sigma_S$ commutes with every generator of $C_W$.

We know that $\mathcal P_{\mathcal{A}}=\mathcal P_{\mathcal D}\cup\{W\}$. Every subgroup in $\mathcal P_{\mathcal D}$ is proper in $W$, and conjugation by any $w\in W$ fixes $W$ and sends proper subgroups to proper subgroups. Therefore,
\[
 w\mathcal P_{\mathcal{A}}w^{-1}=\mathcal P_{\mathcal{A}}
 \quad\Longleftrightarrow\quad
 w\mathcal P_{\mathcal D}w^{-1}=\mathcal P_{\mathcal D}.
\]
Moreover, since $\omega_S$ is central, we have $(w\omega_S)\mathcal P_{\mathcal D}(w\omega_S)^{-1} =w\mathcal P_{\mathcal D}w^{-1}$ for every $w\in W$.

Every subcollection of $\mathcal{A}$ is either $\mathcal{U}$ or $\mathcal{U}\cup\{S\}$ for some $\mathcal{U}\subseteq\mathcal D$. Since every member of $\mathcal D$ is properly contained in $S$, the ordering used in \eqref{eq:vertices} gives $ v_{\mathcal{U}\cup\{S\}}=v_{\mathcal{U}}\sigma_S$ and $ w_{\mathcal{U}\cup\{S\}}=w_{\mathcal{U}}\omega_S$. Thus, Lemma~\ref{lem:vertices} shows that every vertex of $Q_{\mathcal{A}}$ has one of the forms $v_{\mathcal{U}}$ and $v_{\mathcal{U}}\sigma_S$, where $\mathcal{U}\subseteq\mathcal D$. Applying Theorem~\ref{thm:criterion} and the preceding identities, we obtain
\[
 v_{\mathcal{U}}\in K_{\mathcal{A}}
 \quad\Longleftrightarrow\quad
 v_{\mathcal{U}}\in K_{\mathcal D}
 \quad\Longleftrightarrow\quad
 v_{\mathcal{U}}\sigma_S\in K_{\mathcal{A}}.
\]
Consequently, $K_{\mathcal{A}}=K_{\mathcal D}\cup K_{\mathcal D}\sigma_S$.

For $\mathcal{U},\mathcal{V}\subseteq\mathcal D$, the subsets $\mathcal{U}$ and $\mathcal{V}\cup\{S\}$ are distinct. Lemma~\ref{lem:vertices} therefore gives $v_{\mathcal{U}}\ne v_{\mathcal{V}\cup\{S\}}=v_{\mathcal{V}}\sigma_S$. Hence, the above union is disjoint. In particular, $\sigma_S\notin K_{\mathcal D}$.

Since $\sigma_S$ is central in $C_W$ and has order two, multiplication defines a homomorphism
\[
 K_{\mathcal D}\times\langle\sigma_S\rangle
 \longrightarrow K_{\mathcal{A}},
 \qquad
 (h,\sigma_S^\varepsilon)\longmapsto h\sigma_S^\varepsilon.
\]
The disjoint decomposition above shows that every element of $K_{\mathcal{A}}$ has a unique expression $h\sigma_S^\varepsilon$, with $h\in K_{\mathcal D}$ and $\varepsilon\in\{0,1\}$. Thus, this homomorphism is an isomorphism, and consequently, $K_{\mathcal{A}} \cong K_{\mathcal D}\times\mathbb{Z}/2\mathbb{Z}$. 
\end{proof}

\subsection{Reducible Coxeter systems}
By Theorem~\ref{thm:localization}, every finite subgroup of $C_W$ has a conjugate contained in $C_T$ for some spherical subset $T\subseteq S$. Since $W_T$ has only finitely many irreducible components, it suffices to consider finite products when reducing the classification to irreducible Coxeter systems.

If $(W,S)=\prod_{i=1}^r(W_i,S_i)$, then each connected spherical subset of $S$ lies in one $S_i$, and generators from distinct components commute. Therefore, the defining presentation gives $C_W\cong\prod_{i=1}^r C_{W_i}$. Under this identification, $g_W$ is the product of the maps $g_{W_i}$, so also $PC_W\cong\prod_{i=1}^r PC_{W_i}$. Neither the $W_i$ nor the $S_i$ need be finite.

\begin{proposition}\label{prop:product}
Let $(W,S)=\prod_{i=1}^r(W_i,S_i)$ be a finite product of Coxeter systems, and let $F$ be a finite group. Then $F$ embeds into $C_W$ if and only if $F$ embeds into $\prod_i F_i$ for some finite subgroups $F_i\leq C_{W_i}$.
\end{proposition}
\begin{proof}
For the forward implication, identify $F$ with a subgroup of $\prod_{i=1}^r C_{W_i}$ using the preceding isomorphism. For each $i$, define the projection $\operatorname{pr}_i:\prod_{j=1}^r C_{W_j}\longrightarrow C_{W_i}$ onto the $i$-th factor, that is, $\operatorname{pr}_i(h_1,\ldots,h_r)=h_i$. Let $F_i=\operatorname{pr}_i(F)$. Then each $F_i$ is finite, and $F\leq\prod_{i=1}^r F_i$.

For the reverse, compose the given embedding with the inclusion $\prod_{i=1}^r F_i \hookrightarrow \prod_{i=1}^r C_{W_i}\cong C_W$.
\end{proof}

\begin{remark}
A finite subgroup of a direct product need not itself be a product of subgroups of the factors. However, a maximal finite subgroup of a finite direct product is exactly a product of maximal finite subgroups of the factors. Indeed, it is contained in the finite product of its coordinate projections, and maximality forces equality and maximality of each projection. The converse follows by projecting a finite supergroup.

If the Coxeter diagram has infinitely many components, the corresponding decomposition of $C_W$ is a restricted direct product: each element has only finitely many nontrivial coordinates. Every finite subgroup uses only finitely many components, so the finite product argument still applies.
\end{remark}

\section{Classification in type A}\label{sec:typeA}
Let $(W,S)$ be a Coxeter system of type $A_{n-1}$, and write $S=\{s_1,\ldots,s_{n-1}\}$, so that $s_i$ and $s_j$ are adjacent in the Coxeter diagram if and only if $|i-j|=1$. We identify $W$ with $S_n$ via the isomorphism $ s_i\mapsto(i\ i+1)$ for $1\leq i\leq n-1$; see \cite[Example~1.2.3 and Proposition~1.5.4]{BjornerBrenti}.

\bigskip

\begin{figure}[htbp]
\centering
\begin{tikzpicture}[coxeter,node distance=6mm]
\node (a) {$s_1$};
\node (b) [right=of a] {$s_2$};
\node (c) [right=of b] {$\cdots$};
\node (d) [right=of c] {$s_{n}$};
\draw (a)--(b)--(c)--(d);
\end{tikzpicture}
\caption{Coxeter diagram of type $A_n$.}
\label{fig:typeA}
\end{figure}

In type $A_{n-1}$, the classification can be obtained without an analysis of cube stabilizers. By Lemma~\ref{lem:pure}, every finite subgroup of $J_n=C_{W(A_{n-1})}$ is a $2$-group and embeds into $S_n$, hence into a Sylow $2$-subgroup of $S_n$. Therefore, it suffices to construct a subgroup of $J_n$ isomorphic to a Sylow $2$-subgroup of $S_n$. We give such a construction using generators corresponding to dyadic intervals.

We use the generators $s_{p,q}$ of $J_n=C_{W(A_{n-1})}$ and the canonical map $\Sigma_n$ introduced in Section~\ref{sec:introduction}. By \eqref{eq:ordinarypresentation}, subgroups supported on disjoint intervals commute. Recall that $\nu_p(n)$ is the exponent of the largest power of $p$ that divides $n$.

\begin{lemma}\label{lem:sylow}
For $a\geq0$, the group $\mathcal W_a$ is a Sylow $2$-subgroup of $S_{2^a}$ and has order $2^{2^a-1}$. If
\[
 n=2^{a_1}+\cdots+2^{a_k},\qquad a_1>\cdots>a_k\geq0,
\]
is the binary expansion, then every Sylow $2$-subgroup of $S_n$ is isomorphic to $\prod_{j=1}^k\mathcal W_{a_j}$.
\end{lemma}
\begin{proof}
The recurrence $|\mathcal W_a|=2|\mathcal W_{a-1}|^2$ gives the order formula. Realize $\mathcal W_a$ by the independent actions of two copies of $\mathcal W_{a-1}$ on two halves of a set of size $2^a$, together with their exchange. By Legendre's formula, we have
\[
 \nu_2((2^a)!)=2^{a-1}+\cdots+1=2^a-1.
\]
Thus, $|\mathcal W_a|$ is the order of a Sylow $2$-subgroup of $S_{2^a}$, and hence $\mathcal W_a$ is a Sylow $2$-subgroup of $S_{2^a}$.

Now divide $\{1,\ldots,n\}$ into disjoint consecutive blocks of sizes $2^{a_1},\ldots,2^{a_k}$. Let the copy of $\mathcal W_{a_i}$ act on the block of size $2^{a_i}$. Since these blocks are disjoint, we obtain an embedding
\[
\mathcal W_{a_1}\times\cdots\times \mathcal W_{a_k}\hookrightarrow S_n.
\]
Moreover,
\[
|\mathcal W_{a_1}\times\cdots\times \mathcal W_{a_k}|
=
\prod_{i=1}^k |\mathcal W_{a_i}|
=
\prod_{i=1}^k 2^{2^{a_i}-1}
=
2^{\sum_{i=1}^k(2^{a_i}-1)}
=
2^{n-k}.
\]
By Legendre's formula again, we have $\nu_2(n!)=n-k$. It follows that $\mathcal W_{a_1}\times\cdots\times \mathcal W_{a_k}$ has the order of a Sylow $2$-subgroup of $S_n$. Therefore, $\mathcal W_{a_1}\times\cdots\times \mathcal W_{a_k}$ is a Sylow $2$-subgroup of $S_n$.
\end{proof}

Let $\Omega$ be a consecutive interval of $2^a$ integers, where $a\geq0$. We call $\Omega$ and all intervals obtained by repeatedly dividing it into two consecutive intervals of equal size the \emph{dyadic subintervals} of $\Omega$. For $X\subseteq\{1,\ldots,n\}$, let $S_X\leq S_n$ be the subgroup fixing every element outside $X$.

\begin{lemma}\label{lem:dyadic}
Let $\Omega=\{p,\ldots,q\}\subseteq\{1,\ldots,n\}$, where $q-p+1=2^a$ for some $a\geq0$. Then there exists a subgroup $L_\Omega \leq \langle s_{i,j}:p\leq i<j\leq q\rangle$ isomorphic to $\mathcal W_a$ such that the restriction of $\Sigma_n$ to $L_\Omega$ is injective and $\Sigma_n(L_\Omega)$ is a Sylow $2$-subgroup of $S_\Omega$.
\end{lemma}
\begin{proof}
We proceed by induction on $a$. For $a=0$, take $L_\Omega=\{1\}$. Suppose that $a\geq1$, and let $\Omega_-$ and $\Omega_+$ be the left and right halves of $\Omega$. Choose $L_-$ satisfying the assertion for $\Omega_-$, and let $L_+=s_{p,q}L_-s_{p,q}$. The relation $s_{p,q}s_{i,j}s_{p,q}=s_{p+q-j,p+q-i}$ shows that $L_+$ is contained in the subgroup generated by the $s_{i,j}$ with $\{i,\ldots,j\}\subseteq\Omega_+$.

Since $\Omega_-$ and $\Omega_+$ are disjoint, the defining relations imply that every element of $L_-$ commutes with every element of $L_+$. Their images $\Sigma_n(L_-)$ and $\Sigma_n(L_+)$ fix every element outside $\Omega_-$ and $\Omega_+$, respectively, and hence have trivial intersection. Since $\Sigma_n$ is injective on $L_-$ by induction, we also have $L_-\cap L_+=\{1\}$. Therefore, $L_-L_+\cong\mathcal W_{a-1}\times\mathcal W_{a-1}$. Conjugation by $s_{p,q}$ exchanges the two factors. Moreover, $s_{p,q}\notin L_-L_+$: every element of $\Sigma_n(L_-L_+)$ preserves each half, whereas $\Sigma_n(s_{p,q})$ exchanges them. Since $s_{p,q}^2=1$, it follows that
\[
 L_\Omega:=\langle L_-L_+,s_{p,q}\rangle
 \cong(\mathcal W_{a-1}\times\mathcal W_{a-1})
       \rtimes\mathbb Z/2\mathbb Z
 =\mathcal W_a.
\]

In particular, $L_\Omega$ is finite, so Lemma~\ref{lem:pure} implies that $\Sigma_n$ is injective on $L_\Omega$. Thus, $\Sigma_n(L_\Omega)\leq S_\Omega$ has order $|\mathcal W_a|$, which is the order of a Sylow $2$-subgroup of $S_\Omega$ by Lemma~\ref{lem:sylow}. This proves the assertion.
\end{proof}

We will also use this construction in subdiagrams of type $A$. Let $T=\{t_1,\ldots,t_{m-1}\}\subseteq S$ be a subdiagram of type $A_{m-1}$, with edges $\{t_i,t_{i+1}\}$ for $1\leq i\leq m-2$. The identification $J_m\cong C_{W_T}$ and Proposition~\ref{prop:parabolic} give an embedding
\[
 \iota:J_m\hookrightarrow C_W,
 \qquad
 s_{p,q}\longmapsto
 \sigma_{\{t_p,\ldots,t_{q-1}\}}.
\]
Identify $W_T$ with $S_m$ by sending $t_i$ to $(i\ i+1)$. Under this identification, $g_W(\iota(x))$ corresponds to $\Sigma_m(x)$ for every $x\in J_m$. Thus, the subgroups constructed in Lemma~\ref{lem:dyadic} can be embedded into $C_W$ through $\iota$, with their generators given by the displayed correspondence.

Now we are ready to prove Theorem~A.

\begin{theorem}\label{thm:typeA}
For $n\geq2$ and a finite group $F$, the following are equivalent:
\begin{enumerate}
\item $F\hookrightarrow J_n$;
\item $F\hookrightarrow P_n$;
\item $F\hookrightarrow\prod_j\mathcal W_{a_j}$, where $n=\sum_j2^{a_j}$ is the binary expansion with distinct exponents.
\end{enumerate}
\end{theorem}
\begin{proof}
If $F\leq J_n$ is finite, Lemma~\ref{lem:pure} identifies it with a $2$-subgroup of $S_n$. Such a subgroup lies in a Sylow $2$-subgroup, proving $(1)\Rightarrow(2)$. Lemma~\ref{lem:sylow} gives $(2)\Leftrightarrow(3)$.

For the converse, partition $\{1,\ldots,n\}$ into consecutive intervals $B_1,\ldots,B_k$ with $|B_j|=2^{a_j}$. By Lemma~\ref{lem:dyadic}, for each $j$ there is a subgroup $L_j\cong\mathcal W_{a_j}$ generated by elements $s_{p,q}$ with $\{p,\ldots,q\}\subseteq B_j$. Since generators corresponding to disjoint intervals commute, multiplication defines a homomorphism
\[
 \Phi:\prod_{j=1}^k L_j\longrightarrow J_n,
 \qquad
 (\ell_1,\ldots,\ell_k)\longmapsto\ell_1\cdots\ell_k.
\]

To prove injectivity, suppose that $\ell_1\cdots\ell_k=1$, where $\ell_j\in L_j$. Applying $\Sigma_n$ gives $\Sigma_n(\ell_1)\cdots\Sigma_n(\ell_k)=1$. For each $j$, every $\Sigma_n(\ell_i)$ with $i\ne j$ fixes every element of $B_j$. Thus, the equality implies that $\Sigma_n(\ell_j)$ fixes every element of $B_j$. It also fixes every element outside $B_j$, so $\Sigma_n(\ell_j)=1$. Since the restriction of $\Sigma_n$ to $L_j$ is injective by Lemma~\ref{lem:dyadic}, we obtain $\ell_j=1$ for every $j$. Hence, $\Phi$ is injective, proving that $\prod_{j=1}^k\mathcal W_{a_j}$ embeds into $J_n$. This proves $(3)\Rightarrow(1)$.
\end{proof}

\section{Classification in type B}\label{sec:typeB}
We realize $W(B_n)$ as the group of signed permutations, that is, the bijections $w$ of $\{\pm 1,\ldots,\pm n\}$ satisfying $w(-i)=-w(i)$ for $1\leq i\leq n$; see \cite[Section~8.1 and Proposition~8.1.3]{BjornerBrenti}. We identify $S_n$ with the subgroup preserving $\{1,\ldots,n\}$. Let $V_n$ be the subgroup of sign changes, consisting of those $w$ for which $w(i)\in\{i,-i\}$ for every $i$. Then
\[
 W(B_n)=V_n\rtimes S_n,
 \qquad
 V_n\cong(\mathbb Z/2\mathbb Z)^n,
\]
where $S_n$ acts by permuting the $n$ sign choices. The projection onto $S_n$ is given by
\[
 \pi:W(B_n)\longrightarrow S_n,
 \qquad
 \pi(w)(i)=|w(i)|.
\]

Write $S=\{s_1,\ldots,s_n\}$ for the set of generators. For $n\geq2$, we number them so that the Coxeter diagram has precisely the edges $\{s_i,s_{i+1}\}$, $1\leq i\leq n-1$, with label $4$ for $i=1$ and label $3$ for $i\geq2$.

\begin{figure}[htbp]
\centering
\begin{tikzpicture}[coxeter]
\node (a) {$s_1$};
\node (b) [right=of a] {$s_2$};
\node (c) [right=of b] {$s_3$};
\node (d) [right=of c] {$\cdots$};
\node (e) [right=of d] {$s_n$};

\draw (a.east)--node[above] {$4$} (b.west);
\draw (b.east)--(c.west);
\draw (c.east)--(d.west);
\draw (d.east)--(e.west);
\end{tikzpicture}
\caption{Coxeter diagram of type $B_n$.}
\label{fig:typeB}
\end{figure}

Under the above realization, $s_1$ interchanges $1$ and $-1$, while, for $2\leq i\leq n$, $s_i$ interchanges $i-1$ with $i$ and $-(i-1)$ with $-i$. Each generator fixes all remaining elements.

We denote $T_b=\{s_1,\ldots,s_b\}$ for $1\leq b\leq n$ and $A_{p,q}=\{s_{p+1},\ldots,s_q\}$ for $1\leq p<q\leq n$. Every nonempty connected subset of $S$ has one of these forms. To $A_{p,q}$ we associate the interval $\{p,\ldots,q\}$ of integers.

The longest element $\omega_{T_b}$ interchanges $i$ and $-i$ for $1\leq i\leq b$ and fixes all remaining elements. This follows from the length formula in \cite[Proposition~8.1.1 and equation~(8.2)]{BjornerBrenti}. In particular, $\omega_{T_b}$ is central in $W_{T_b}$.

The longest element $\omega_{A_{p,q}}$ belongs to the subgroup $S_n$ and  reverses $\{p,\ldots,q\}$, fixing every element of $\{1,\ldots,n\}$ outside this interval; see \cite[Section~2.3, p.~37]{BjornerBrenti}.

For $P\leq S_n$, let $c(P)$ be the number of its orbits, including fixed points, and let
\[
 E_B(P)=\{\mathrm d_E:E\subseteq\{1,\ldots,n\},\ p(E)=E\text{ for every }p\in P\}\leq V_n,
\]
where $\mathrm d_E\in V_n$ is defined, for $1\leq i\leq n$, by
\[
 \mathrm d_E(i)=
 \begin{cases}
  -i,&i\in E,\\
  i,&i\notin E.
 \end{cases}
\]
Then $E_B(P)\cong(\mathbb{Z}/2\mathbb{Z})^{c(P)}$, and $P$ centralizes $E_B(P)$.

\begin{theorem}\label{thm:typeB}
Let $n\geq1$ and let $F$ be finite. The following are equivalent:
\begin{enumerate}
\item $F\hookrightarrow C_{W(B_n)}$;
\item $F\hookrightarrow(\mathbb{Z}/2\mathbb{Z})^{c(P)}\times P$ for some $2$-subgroup $P\leq S_n$;
\item $F\hookrightarrow H$ for some $H\in\mathscr B_n$.
\end{enumerate}
\end{theorem}
\begin{proof}
If $n=1$, then $C_{W(B_1)}\cong\mathbb Z/2\mathbb Z$, and all three conditions are equivalent to $F\hookrightarrow\mathbb Z/2\mathbb Z$. We therefore assume that $n\geq2$.

\emph{$(1)\Rightarrow(2)$.} Identify $F$ with its image in $C_{W(B_n)}$. By the argument in Lemma~\ref{lem:pure} regarding fixed points, after conjugation we may assume that $F$ stabilizes $Q_{\mathcal A}$ for some compatible family $\mathcal A$. Since $1\in V(Q_{\mathcal A})$, every element of $F$ is a vertex of $Q_{\mathcal A}$. Moreover, Lemma~\ref{lem:pure} shows that $F$ is a $2$-group and that $g_{W(B_n)}$ is injective on $F$.

Write
\[
 \{b:1\leq b<n,\ T_b\in\mathcal A\}
 =\{b_1<\cdots<b_t\},
 \qquad b_0=0,\quad b_{t+1}=n,
\]
and let
\[
 \Lambda_i=\{b_{i-1}+1,\ldots,b_i\},
 \qquad 1\leq i\leq t+1.
\]
For every $A_{p,q}\in\mathcal A$, the interval $\{p,\ldots,q\}$ is contained in one $\Lambda_i$. Otherwise, $p\leq b_i<q$ for some $i$. Then $A_{p,q}$ and $T_{b_i}$ are not nested. They are not orthogonal either, since $s_{b_i}\in T_{b_i}$ is adjacent to $s_{b_i+1}\in A_{p,q}$. This contradicts compatibility.

Let
\[
 D_\Lambda=
 \langle\mathrm d_{\Lambda_1},\ldots,\mathrm d_{\Lambda_{t+1}}\rangle
 \leq V_n,
 \qquad
 S_\Lambda=\prod_{i=1}^{t+1}S_{\Lambda_i}\leq S_n.
\]
Every element of $S_\Lambda$ preserves each $\Lambda_i$ and hence commutes with every element of $D_\Lambda$. Thus, their product is the subgroup $D_\Lambda\times S_\Lambda$ of $W(B_n)$. If $A_{p,q}\in\mathcal A$, then $\omega_{A_{p,q}}$ belongs to $S_\Lambda$. If $T_b\in\mathcal A$, then $\{1,\ldots,b\}$ is a union of the intervals $\Lambda_i$, so $\omega_{T_b}=\mathrm d_{\{1,\ldots,b\}}$ belongs to $D_\Lambda$. Consequently, Lemma~\ref{lem:vertices} gives $g_{W(B_n)}(V(Q_{\mathcal A})) \subseteq D_\Lambda\times S_\Lambda$.

Let $P=\pi(g_{W(B_n)}(F))$. Then $P$ is a $2$-subgroup of $S_\Lambda$, so every $P$-orbit is contained in one $\Lambda_i$. Hence, every union of the $\Lambda_i$ is $P$-invariant, and $D_\Lambda\leq E_B(P)$ by the definition of $E_B(P)$. It follows that
\[
 F\cong g_{W(B_n)}(F)
 \leq D_\Lambda\times P
 \leq E_B(P)\times P
 \cong(\mathbb{Z}/2\mathbb{Z})^{c(P)}\times P.
\]
This proves $(1)\Rightarrow(2)$.

\emph{$(2)\Rightarrow(3)$.} Let $P$ be as in (2), and let $\Omega_1,\ldots,\Omega_k$ be its orbits on $\{1,\ldots,n\}$. Since $P$ is a $2$-group, the orbit-stabilizer theorem gives $|\Omega_j|=2^{a_j}$ for some $a_j\geq0$. In particular, $n=\sum_{j=1}^k2^{a_j}$. For each $j$, let $\rho_j:P\to S_{\Omega_j}$ be the restriction homomorphism. The product of these homomorphisms is injective, since the orbits partition $\{1,\ldots,n\}$. Each $\rho_j(P)$ is a $2$-subgroup of $S_{\Omega_j}\cong S_{2^{a_j}}$, so it embeds into $\mathcal W_{a_j}$ by Lemma~\ref{lem:sylow}. Since $k=c(P)$, we obtain
\[
 (\mathbb{Z}/2\mathbb{Z})^{c(P)}\times P
 \hookrightarrow\prod_{j=1}^k(\mathbb{Z}/2\mathbb{Z}\times\mathcal W_{a_j})
 \in\mathscr B_n.
\]
This proves $(2)\Rightarrow(3)$.

\emph{$(3)\Rightarrow(1)$.} It suffices to realize every group in $\mathscr B_n$ as a subgroup of $C_{W(B_n)}$. Fix $a_1,\ldots,a_k\geq0$ with $n=\sum_{j=1}^k2^{a_j}$, and partition $\{1,\ldots,n\}$ into consecutive intervals $\Omega_j=\{p_j,\ldots,q_j\}$, where $|\Omega_j|=2^{a_j}$. For $a_j\geq1$, let $L_{\Omega_j}\leq J_n$ be the subgroup constructed in Lemma~\ref{lem:dyadic}. The subset $\{s_2,\ldots,s_n\}$ of the type $B_n$ diagram has type $A_{n-1}$. Hence, Proposition~\ref{prop:parabolic} gives an embedding
\[
 \iota:J_n\hookrightarrow C_{W(B_n)},
 \qquad
 s_{r,s}\longmapsto\sigma_{A_{r,s}}.
\]
Define $L_j=\iota(L_{\Omega_j})$. Then $L_j\leq\langle\sigma_{A_{r,s}}:p_j\leq r<s\leq q_j\rangle$ and $L_j\cong\mathcal W_{a_j}$. Under the identification $S_n\leq W(B_n)$, we have $g_{W(B_n)}(\iota(x))=\Sigma_n(x)$ for each $x\in J_n$. Thus, Lemma~\ref{lem:dyadic} implies that $g_{W(B_n)}$ is injective on $L_j$ and that $g_{W(B_n)}(L_j)\leq S_{\Omega_j}$. For $a_j=0$, set $L_j=\{1\}$.

Let $z_j=\sigma_{T_{q_j}}\sigma_{T_{p_j-1}}$ with convention $\sigma_{T_0}:=1$. Since $\omega_{T_b}$ is central in $W_{T_b}$, the nesting relations imply that the generators $\sigma_{T_b}$ commute with one another. Therefore, $z_j^2=1$, and $g_{W(B_n)}(z_j)=\mathrm d_{\Omega_j}$.

For each positive $b\in\{p_i-1,q_i\}$, the partition implies that $A_{r,s}$ is contained in or orthogonal to $T_b$ whenever $p_j\leq r<s\leq q_j$. The defining relations, together with the centrality of $\omega_{T_b}$ in $W_{T_b}$, show that $\sigma_{T_b}$ commutes with every such $\sigma_{A_{r,s}}$. Hence, $z_i$ commutes with every element of $L_j$ for all $i,j$. The subgroups $L_i$ and $L_j$ also commute for $i\ne j$, since their indexing subsets are orthogonal, and the elements $z_i$ commute with one another.

Consequently, multiplication defines a homomorphism
\[
 \Phi:\prod_{j=1}^k(\mathbb{Z}/2\mathbb{Z}\times L_j)
 \longrightarrow C_{W(B_n)},
 \qquad
 ((\varepsilon_j,\ell_j))_{j=1}^k
 \longmapsto\prod_{j=1}^k z_j^{\varepsilon_j}\ell_j,
\]
where $\varepsilon_j\in\{0,1\}$. Suppose that $\prod_{j=1}^k z_j^{\varepsilon_j}\ell_j=1$. Applying $\pi\circ g_{W(B_n)}$ gives $\prod_{j=1}^k g_{W(B_n)}(\ell_j)=1$. Each $g_{W(B_n)}(\ell_j)$ fixes every element outside $\Omega_j$. Restricting the equality to each $\Omega_j$ gives $g_{W(B_n)}(\ell_j)=1$, and injectivity on $L_j$ implies $\ell_j=1$. Applying $g_{W(B_n)}$ to the remaining equality gives $\prod_{j=1}^k\mathrm d_{\Omega_j}^{\varepsilon_j}=1$. Since the intervals $\Omega_j$ are nonempty and pairwise disjoint, every $\varepsilon_j$ is zero. Thus, $\Phi$ is injective. This realizes every group in $\mathscr B_n$ and proves $(3)\Rightarrow(1)$.
\end{proof}
\begin{remark}\label{rem:Bsmall}
Theorem~\ref{thm:typeB} implies that every finite subgroup of $C_{W(B_2)}$ and $C_{W(B_3)}$ is elementary abelian of rank at most two and three, respectively. Indeed, $\mathcal W_0=\{1\}$ and $\mathcal W_1\cong\mathbb{Z}/2\mathbb{Z}$, so every group in $\mathscr B_2$ is isomorphic to $(\mathbb{Z}/2\mathbb{Z})^{2}$, and every group in $\mathscr B_3$ is isomorphic to $(\mathbb{Z}/2\mathbb{Z})^{3}$.

In type $A$, every $2$-subgroup of $S_n$ embeds into $J_n$. For a general Coxeter system $(W,S)$, however, a $2$-subgroup of $W$ need not embed into $C_W$. For example, $W(B_2)$ is the dihedral group of order $8$, whereas every finite subgroup of $C_{W(B_2)}$ has order at most $4$. Therefore, $W(B_2)$ itself does not embed into $C_{W(B_2)}$.
\end{remark}

\section{Classification in type D}\label{sec:typeD}
\subsection{Cube stabilizers and sign changes}
Throughout this section, we assume $n\geq4$. We realize $W(D_n)$ as the subgroup of $W(B_n)$ consisting of those signed permutations $w$ for which the number of indices $i\in\{1,\ldots,n\}$ satisfying $w(i)<0$ is even; see \cite[Section~8.2 and Proposition~8.2.3]{BjornerBrenti}. Write $S=\{s_1,\ldots,s_n\}$ for the set of generators, numbered as follows.

\begin{figure}[htbp]
\centering
\begin{tikzpicture}[coxeter]
\node (c) {$s_3$};
\node (a) [above left=of c] {$s_1$};
\node (b) [below left=of c] {$s_2$};
\node (d) [right=of c] {$\cdots$};
\node (e) [right=of d] {$s_n$};

\draw (a.south east)--(c.north west);
\draw (b.north east)--(c.south west);
\draw (c.east)--(d.west);
\draw (d.east)--(e.west);
\end{tikzpicture}
\caption{Coxeter diagram of type $D_n$.}
\label{fig:typeD}
\end{figure}

As in the type $B$ case, for $E\subseteq\{1,\ldots,n\}$, we write $\mathrm d_E\in W(B_n)$ for the permutation interchanging $i$ and $-i$ for $i\in E$ and fixing all remaining elements. We also write
\[
 \pi:W(D_n)\longrightarrow S_n,
 \qquad
 \pi(w)(i)=|w(i)|,
\]
for the restriction of the projection defined above.

Under this realization, $s_1=\mathrm d_{\{1,2\}}(1\ 2)$, $s_2=(1\ 2)$, and $s_i=(i-1\ i)$ for $3\leq i\leq n$; see \cite[Section~8.2, p.~253]{BjornerBrenti}. The multiplication rule is
\begin{equation}\label{eq:signedmultiply}
 (\mathrm d_Ep)(\mathrm d_Fq)=\mathrm d_{E\triangle p(F)}pq.
\end{equation}
Every nonempty connected subset of $S$ has one of the following forms:
\[
 A^+_{p,q}=\{s_{p+1},\ldots,s_q\}
 \quad(1\leq p<q\leq n),
 \qquad
 A^-_q=\{s_1,s_3,\ldots,s_q\}
 \quad(2\leq q\leq n),
\]
where $A^-_2=\{s_1\}$, or
\[
 \Delta_q=\{s_1,\ldots,s_q\}
 \qquad(3\leq q\leq n).
\]
We call the subsets $A^+_{p,q}$ and $A^-_q$ \emph{$A$-paths}. We call each subset $\Delta_q$ a \emph{prefix}. This terminology includes $\Delta_3$, which has type $A_3$. A prefix $\Delta_q$ is called \emph{even} or \emph{odd} according to the parity of $q$.

For $1\leq p<q\leq n$, let $r_{p,q}\in S_n$ be the permutation sending $i$ to $p+q-i$ for $p\leq i\leq q$ and fixing every element of $\{1,\ldots,n\}$ outside $\{p,\ldots,q\}$. Then
\begin{equation}\label{eq:Dlongest}
\begin{aligned}
 \omega_{A^+_{p,q}}&=r_{p,q},\\
 \omega_{A^-_q}&=\mathrm d_{\{1\}}r_{1,q}\mathrm d_{\{1\}},\\
 \omega_{\Delta_q}&=
 \begin{cases}
  \mathrm d_{\{1,\ldots,q\}},&q\text{ even},\\
  \mathrm d_{\{2,\ldots,q\}},&q\text{ odd}.
 \end{cases}
\end{aligned}
\end{equation}
The formula for $\omega_{\Delta_q}$ follows from the length formula in \cite[Proposition~8.2.1 and equation~(8.18)]{BjornerBrenti}. Conjugation by $\mathrm d_{\{1\}}\in W(B_n)$ exchanges $s_1$ and $s_2$ and fixes all remaining generators. By \eqref{eq:Dlongest}, opposition on an even prefix is trivial; on an odd prefix, it exchanges $s_1$ and $s_2$ and fixes all remaining vertices of that prefix.

For $P\leq S_X$, let $c(P)$ denote the number of its orbits on $X$, including fixed points, and define
\[
 E_X(P)=
 \{\mathrm d_E:E\subseteq X,\ p(E)=E
 \text{ for every }p\in P,\ |E|\text{ even}\}.
\]
This is a subgroup of $W(D_n)$, and every element of $P$ commutes with every element of $E_X(P)$. For $P\leq S_n$, we write $E_D(P)=E_{\{1,\ldots,n\}}(P)$.

If $P$ is a $2$-group, then
\begin{equation}\label{eq:even-sign-rank}
 E_X(P)\cong
 \begin{cases}
  (\mathbb{Z}/2\mathbb{Z})^{c(P)},&
  \text{if $P$ has no fixed point on }X,\\
  (\mathbb{Z}/2\mathbb{Z})^{c(P)-1},&
  \text{if $P$ has a fixed point on }X.
 \end{cases}
\end{equation}
Indeed, a subset of $X$ is preserved by every element of $P$ if and only if it is a union of $P$-orbits. Each orbit has cardinality a power of two. If there are no fixed points, all orbits have even cardinality, so every union of orbits has even cardinality. If fixed points occur, such a union has even cardinality exactly when it contains an even number of fixed points. Choose a fixed point $x_0\in X$. After choosing which of the other $c(P)-1$ orbits to include, there is a unique choice of whether to include $x_0$ that makes the total cardinality even. Hence, $E_X(P)\cong(\mathbb{Z}/2\mathbb{Z})^{c(P)-1}$.

We next apply Theorem~\ref{thm:criterion} to the subgroups $g_{W(D_n)}(K_{\mathcal A})\leq W(D_n)$. The following lemma shows that, after possible conjugation in $W(B_n)$, each such subgroup is contained in a direct product of an elementary abelian $2$-group and a $2$-subgroup of $S_n$, possibly with an additional dihedral factor of order eight.

\begin{lemma}\label{lem:Dnormal}
Let $\mathcal A$ be a compatible family. At least one of the following holds.
\begin{enumerate}
\item[(i)] For some $b\in\{1,\mathrm d_{\{1\}}\}$ and $P=\pi(bH_{\mathcal{A}}b^{-1})$, we have
\[
bH_{\mathcal{A}}b^{-1}\leq E_D(P)\times P.
\]
\item[(ii)] The family contains both $\{s_1\},\{s_2\}$ and an odd prefix $\Delta_{m_*}$. Let $Y=\{3,\ldots,n\}$ and $U=\{3,\ldots,m_*\}$. There exists a $2$-subgroup $P_Y\leq S_Y$ preserving $U$ such that
\[
H_{\mathcal {A}}\leq H_*\times E_Y(P_Y)\times P_Y,
\]
where $H_*=\langle(1\ 2),\mathrm d_{\{1,2\}},\mathrm d_{\{2\}\cup U}\rangle\cong D_8\cong \mathcal W_2$.
\end{enumerate}
Alternative (i) holds unless $\{s_1\},\{s_2\}\in\mathcal A$ and $\mathcal A$ contains an odd prefix.
\end{lemma}
\begin{proof}
By Theorem~\ref{thm:criterion}, $H_{\mathcal A}$ is a finite $2$-subgroup of $W(D_n)$. Every element of $H_{\mathcal A}$ normalizes $\mathcal P_{\mathcal A}$ and can be written as $w_{\mathcal U}$ for some $\mathcal U\subseteq\mathcal A$.

We distinguish three cases according to whether $\mathcal A$ contains both $\{s_1\}$ and $\{s_2\}$ and whether it contains an odd prefix. We prove (i) in Cases~1 and~2, and (ii) in Case~3.

\emph{Case 1: at least one of $\{s_1\},\{s_2\}$ does not belong to $\mathcal A$.} Take $b=\mathrm d_{\{1\}}$ if an $A^-_q$ occurs, and $b=1$ otherwise. We claim that, for every $A$-path $I\in\mathcal A$, $bW_Ib^{-1}=S_\Omega$ for some $\Omega=\{p,\ldots,q\}$ with $1\leq p<q\leq n$. If $b=1$, all paths are already of the form $A^+_{p,q}$. If $b=\mathrm d_{\{1\}}$, conjugation exchanges $s_1,s_2$ and takes each $A^-_q$ to $A^+_{1,q}$. However, if $A^-_q\in\mathcal A$, then no $A^+_{1,r}$ can belong to $\mathcal A$. Indeed, compatibility of these two subsets would force $q=r=2$, so both $\{s_1\}$ and $\{s_2\}$ would belong to $\mathcal A$, contrary to the assumption of this case. The $A$-paths containing neither $s_1$ nor $s_2$ are fixed by conjugation by $b$. This proves the claim.

Let $\mathcal P'=b\mathcal P_{\mathcal A}b^{-1}$, $H'_{\mathcal A}=bH_{\mathcal A}b^{-1}$, and let $P=\pi(H'_{\mathcal A})$. Since $W(D_n)$ is normal in $W(B_n)$, we have $H'_{\mathcal A}\leq W(D_n)$, and $P$ is a $2$-subgroup of $S_n$. To prove that $H'_{\mathcal A}\leq E_D(P)\times P$, fix an arbitrary element $w'=\mathrm d_Ep\in H'_{\mathcal A}$, where $E\subseteq\{1,\ldots,n\}$ and $p=\pi(w')\in P$.

We first show that conjugation by $w'$ permutes the subgroups $S_\Omega$ in $\mathcal P'$ among themselves, and likewise the subgroups $W_{\Delta_q}$. To this end, we compare their intersections with $\ker\pi$. For each subgroup $S_\Omega$ obtained above, we have $S_\Omega\cap\ker\pi=\{1\}$ since $\pi(p)=p$ for every $p\in S_\Omega$. For each prefix $\Delta_q\in\mathcal A$, we have $bW_{\Delta_q}b^{-1}=W_{\Delta_q}$. This subgroup contains the non-trivial element $s_1s_2=\mathrm d_{\{1,2\}}\in\ker\pi$, and hence $W_{\Delta_q}\cap\ker\pi\neq\{1\}$.

Since $\ker\pi$ is normal in $W(D_n)$, conjugation by $w'$ preserves whether a subgroup has trivial intersection with $\ker\pi$. As $w'$ normalizes $\mathcal P'$, the subgroup $w'S_\Omega(w')^{-1}$ belongs to $\mathcal P'$. Its intersection with $\ker\pi$ is trivial, so it must be another subgroup $S_{\Omega'}$. Thus, conjugation by $w'$ permutes the subgroups $S_\Omega$ in $\mathcal P'$.

Suppose that $w'S_\Omega(w')^{-1}=S_{\Omega'}$. We next show that, for every subgroup $S_\Omega\in\mathcal P'$, either $\Omega\subseteq E$ or $\Omega\cap E=\emptyset$. Applying $\pi$ gives $S_{\Omega'}=pS_\Omega p^{-1}=S_{p(\Omega)}$. Since both $\Omega'$ and $p(\Omega)$ have at least two elements, the equality $S_{\Omega'}=S_{p(\Omega)}$ implies $\Omega'=p(\Omega)$. Using $w'=\mathrm d_Ep$, we have
\[
 \mathrm d_E S_{p(\Omega)}\mathrm d_E
 =w'S_\Omega(w')^{-1}
 =S_{\Omega'}
 =S_{p(\Omega)}.
\]

For distinct $i,j\in p(\Omega)$, this equality implies that $\mathrm d_E(i\ j)\mathrm d_E\in S_n$. On the other hand,
\[
 \mathrm d_E(i\ j)\mathrm d_E\in S_n
 \quad\Longleftrightarrow\quad
 \bigl(i\in E\Longleftrightarrow j\in E\bigr).
\]
Consequently, either $p(\Omega)\subseteq E$ or $p(\Omega)\cap E=\emptyset$. Since $w'$ permutes the subgroups $S_\Omega$ in $\mathcal P'$, the equality $\Omega'=p(\Omega)$ shows that $p$ permutes the corresponding sets $\Omega$. This proves the claim.

We now show that $E$ is preserved by every element of $P$. Let $\rho\in P$. By Theorem~\ref{thm:criterion} and the definition of $H'_{\mathcal A}$, there exists $\mathcal U\subseteq\mathcal A$ such that
\[
 \rho=\pi(bw_{\mathcal U}b^{-1})
      =\pi(w_{\mathcal U}),
\]
where the second equality holds because conjugation by $b$ does not change the image under $\pi$. By \eqref{eq:vertices} and \eqref{eq:Dlongest}, $\rho$ is a product of permutations $r_{i,j}$ with $1\leq i<j\leq n$ and $S_{\{i,\ldots,j\}}\in\mathcal P'$; the longest element of each prefix has trivial image under $\pi$. Each such $r_{i,j}$ preserves $E$, since $\{i,\ldots,j\}$ is contained in $E$ or disjoint from $E$, and $r_{i,j}$ fixes every element outside this set. Therefore, their product $\rho$ preserves $E$.

To complete the proof of (i), we verify that $\mathrm d_E\in E_D(P)$ and that $E_D(P)P$ is a direct product. Since $w'=\mathrm d_Ep\in W(D_n)$, the set $E$ has even cardinality. Together with its invariance under $P$, this gives $\mathrm d_E\in E_D(P)$. By the definition of $E_D(P)$, its elements commute with every element of $P$. Moreover, $E_D(P)\cap P=\{1\}$, since $E_D(P)\subseteq\ker\pi$ and $\pi$ is injective on $P$. Thus, multiplication identifies $E_D(P)P$ with $E_D(P)\times P$. Since $w'=\mathrm d_Ep$ was arbitrary, we obtain $ H'_{\mathcal A}\leq E_D(P)\times P$, proving (i).

For the remaining two cases, assume that $\{s_1\},\{s_2\}\in\mathcal A$. Let
\[
 Z=\{1,2\},\qquad Y=\{3,\ldots,n\},\qquad
 \tau=(1\ 2),\qquad U_m=\{3,\ldots,m\}
 \quad(3\leq m\leq n).
\]

By the compatibility with $\{s_1\}$ and $\{s_2\}$, every other $A$-path in $\mathcal A$ should be contained in $\{s_4,\ldots,s_n\}$. Thus, each such member has the form $I=A^+_{p,q}$ with $3\leq p<q\leq n$. For each prefix $\Delta_m\in\mathcal A$, compatibility implies that either $I\subseteq\Delta_m$ or $I\perp\Delta_m$. In the first case, $\{p,\ldots,q\}\subseteq U_m$; in the second, $\{p,\ldots,q\}\cap U_m=\emptyset$. Since $r_{p,q}$ reverses $\{p,\ldots,q\}$ and fixes every element outside it, we obtain
\begin{equation}\label{eq:preserveUm}
 r_{p,q}(U_m)=U_m
 \qquad(\Delta_m\in\mathcal A).
\end{equation}

By the formulas for the longest elements, we have
\[
 g_{W(D_n)}(\sigma_{\{s_1\}})=\mathrm d_Z\tau,
 \qquad
 g_{W(D_n)}(\sigma_{\{s_2\}})=\tau,
\]
and
\[
 g_{W(D_n)}(\sigma_{\Delta_m})=\omega_{\Delta_m}
 =
 \begin{cases}
  \mathrm d_{Z\cup U_m},&m\text{ even},\\
  \mathrm d_{\{2\}\cup U_m},&m\text{ odd}.
 \end{cases}
\]

\emph{Case 2: $\{s_1\},\{s_2\}\in\mathcal A$, and $\mathcal A$ contains no odd prefix.} Let $P=\pi(H_{\mathcal A})$. For every $J\in\mathcal A$, the permutation $\pi(\omega_J)$ is either $1$, $\tau$, or $\omega_I$ for an $A$-path $I\in\mathcal A$ contained in $\{s_4,\ldots,s_n\}$. The permutation $\tau$ preserves $Z$ and fixes $Y$ pointwise. Each such $\omega_I$ fixes $Z$ pointwise and preserves every $U_m$ with $\Delta_m\in\mathcal A$ by \eqref{eq:preserveUm}. Thus, every $\pi(\omega_J)$ preserves $Z$ and each $Z\cup U_m$. Since every element of $P$ is a product of these permutations, these sets are also preserved by every element of $P$.

Let $w=\mathrm d_Ep\in H_{\mathcal A}$, where $E\subseteq\{1,\ldots,n\}$ and $p=\pi(w)\in P$. By Theorem~\ref{thm:criterion}, one can write $w=\omega_{I_1}\cdots\omega_{I_k}$ for some $I_1,\ldots,I_k\in\mathcal A$. For each $j$, write $\omega_{I_j}=\mathrm d_{E_j}p_j$, where $p_j=\pi(\omega_{I_j})$. The formulas for the longest elements given above yield $E_j\in \{\emptyset,Z\} \cup \{Z\cup U_m:\Delta_m\in\mathcal A\}$. Since each $p_j$ preserves every $E_\ell$, the multiplication rule \eqref{eq:signedmultiply} gives
\[
 w=\mathrm d_Ep,\qquad
 E=E_1\triangle\cdots\triangle E_k,\qquad
 p=p_1\cdots p_k=\pi(w)\in P.
\]
Since every $E_j$ is preserved by every element of $P$, the same holds for $E=E_1\triangle\cdots\triangle E_k$. Furthermore, $|E|$ is even because $w=\mathrm d_Ep\in H_{\mathcal A}\leq W(D_n)$. Thus, $\mathrm d_E\in E_D(P)$. Every element of $P$ commutes with every element of $E_D(P)$ by the definition of $E_D(P)$. Moreover, $E_D(P)\cap P=\{1\}$, since $E_D(P)\subseteq\ker\pi$ and $\pi(p)=p$ for every $p\in P$. Therefore, we have $E_D(P)P\cong E_D(P)\times P$. Since $w=\mathrm d_Ep$ was arbitrary, with $\mathrm d_E\in E_D(P)$ and $p\in P$, we obtain $H_{\mathcal A}\leq E_D(P)\times P$. This proves (i) with $b=1$.

\emph{Case 3: $\{s_1\},\{s_2\}\in\mathcal A$, and $\mathcal A$ contains an odd prefix.} We first construct the dihedral subgroup $H_*$ appearing in (ii). Choose an odd prefix $\Delta_{m_*}\in\mathcal A$ and let $U=U_{m_*}$ and $T=\{2\}\cup U$. Since $|T|=m_*-1$ is even, we have $\mathrm d_T\in W(D_n)$. Let $H_*=\langle\tau,\mathrm d_Z,\mathrm d_T\rangle$. The relations
\[
 \tau\mathrm d_Z\tau=\mathrm d_Z,\qquad
 \tau\mathrm d_T\tau=\mathrm d_Z\mathrm d_T,\qquad
 \langle\mathrm d_Z,\mathrm d_T\rangle\cong(\mathbb{Z}/2\mathbb{Z})^{2}
\]
show that $H_*$ is dihedral of order eight.

We next define a $2$-subgroup $P_Y\leq S_Y$ and show that it preserves every $U_m$ with $\Delta_m\in\mathcal A$. The formulas for the longest elements show that $\pi(H_{\mathcal A})$ preserves $Y$, so we can define
\[
 P_Y=\{p|_Y:p\in\pi(H_{\mathcal A})\}\leq S_Y.
\]
As a homomorphic image of $H_{\mathcal A}$, this is a $2$-group. By Theorem~\ref{thm:criterion} and \eqref{eq:Dlongest}, every element of $P_Y$ is a product of permutations $\omega_I\in S_Y$, where $I\in\mathcal A$ is an $A$-path contained in $\{s_4,\ldots,s_n\}$. Each such $\omega_I$ preserves every $U_m$ with $\Delta_m\in\mathcal A$ by \eqref{eq:preserveUm}. Hence, every element of $P_Y$ preserves these sets, including $U$.

We now show that every longest element corresponding to a prefix in $\mathcal A$ belongs to $\langle H_*,E_Y(P_Y)\rangle$. By \eqref{eq:Dlongest},
\[
 \omega_{\Delta_m}=
 \begin{cases}
  \mathrm d_Z\mathrm d_{U_m},&m\text{ even},\\
  \mathrm d_T\mathrm d_{U\triangle U_m},&m\text{ odd}.
 \end{cases}
\]
The first factor in each case belongs to $H_*$. The sets $U_m$ for even $m$ and $U\triangle U_m$ for odd $m$ have even cardinality and are subsets of $Y$ preserved by $P_Y$. Thus, the second factor belongs to $E_Y(P_Y)$. We also have $\omega_{\{s_1\}}=\mathrm d_Z\tau\in H_*$ and $\omega_{\{s_2\}}=\tau\in H_*$.

We claim that $H_{\mathcal A}\leq\langle H_*,E_Y(P_Y),P_Y\rangle$. Let $w\in H_{\mathcal A}$. By Theorem~\ref{thm:criterion}, write $w$ as a product of longest elements indexed by members of $\mathcal A$. For each $A$-path $I\in\mathcal A$ contained in $\{s_4,\ldots,s_n\}$, the permutation $\omega_I$ fixes $Z$ pointwise and preserves every $U_m$ with $\Delta_m\in\mathcal A$ by \eqref{eq:preserveUm}. Therefore, the formulas for the longest elements show that $\omega_I$ commutes with $\omega_{\{s_1\}}$, $\omega_{\{s_2\}}$, and each $\omega_{\Delta_m}$ with $\Delta_m\in\mathcal A$. Moving these factors $\omega_I$ to the right, while preserving their relative order, gives
\[
 w=av,\qquad
 a\in\langle H_*,E_Y(P_Y)\rangle,\qquad v\in S_Y.
\]
Since $\pi(a)\in\langle\tau\rangle$ fixes $Y$ pointwise, we have $v|_Y=\pi(w)|_Y$. Thus $v\in P_Y$, proving the claim.

It remains to show that the three subgroups $H_*$, $E_Y(P_Y)$, and $P_Y$ form a direct product. The groups $E_Y(P_Y)$ and $P_Y$ commute by the definition of $E_Y(P_Y)$, and their intersection is trivial because $\pi$ is trivial on the former and injective on the latter. Both also commute with $H_*$: elements of $E_Y(P_Y)$ change signs only in $Y$, while elements of $P_Y$ fix $Z$ pointwise and preserve $U$. Moreover, we have $H_*\cap\bigl(E_Y(P_Y)P_Y\bigr)=\{1\}$. Indeed, every element $w\in E_Y(P_Y)P_Y$ satisfies $w(1)=1$ and $w(2)=2$. Suppose that $h\in H_*$ also satisfies these equalities. Since $\pi(H_*)=\langle\tau\rangle$ and $\tau=(1\ 2)$, we must have $\pi(h)=1$. Hence,
\[
 h\in\langle\mathrm d_Z,\mathrm d_T\rangle
 =\{1,\mathrm d_Z,\mathrm d_T,\mathrm d_{Z\triangle T}\}.
\]
Each non-trivial element in this list sends at least one of $1$ and $2$ to its negative. Thus $h=1$, proving $H_*\cap\bigl(E_Y(P_Y)P_Y\bigr)=\{1\}$. Consequently, $H_{\mathcal A}\leq H_*\times E_Y(P_Y)\times P_Y$, proving (ii).
\end{proof}

A finite group may embed into groups of both forms appearing in (i) and (ii). In (ii), the group $P_Y$ has a fixed point in $U$. Indeed, $P_Y$ preserves $U$, and each of its orbits in $U$ has cardinality a power of two. Since $|U|$ is odd, at least one of these orbits must consist of a single element.

\subsection{Classification of finite subgroups}
For $a\geq0$, let $B_a=\mathbb{Z}/2\mathbb{Z}\times\mathcal W_a$. Note that $|B_a|=2^{2^a}$.

\begin{theorem}\label{thm:typeD}
For a finite group $F$, the following are equivalent.
\begin{enumerate}
\item $F\hookrightarrow C_{W(D_n)}$.
\item At least one of the following holds:
\begin{enumerate}
\item[(D0)] $F\hookrightarrow E_D(P)\times P$ for some $2$-subgroup $P\leq S_n$;
\item[(D1)] $F\hookrightarrow\mathcal W_2\times(\mathbb{Z}/2\mathbb{Z})^{c(P)}\times P$ for some $2$-subgroup $P\leq S_m$, where $0\leq m\leq n-3$.
\end{enumerate}
For $m=0$, take $P=1$ and $c(P)=0$.
\item At least one of the following holds, with $a_j\geq0$ and empty products allowed:
\begin{align*}
 &F\hookrightarrow\prod_j B_{a_j},&&\sum_j2^{a_j}\leq2\lfloor n/2\rfloor;\tag{B0}\\
 &F\hookrightarrow\mathcal W_2\times\prod_jB_{a_j},&&3+\sum_j2^{a_j}\leq n.\tag{B1}
\end{align*}
\end{enumerate}
\end{theorem}
\begin{proof}
\emph{$(1)\Rightarrow(2)$.} By Theorem~\ref{thm:criterion}, $F$ embeds into $H_{\mathcal A}$ for some compatible family $\mathcal A$. We apply Lemma~\ref{lem:Dnormal} to this group. Alternative (i) immediately gives (D0).

We show that alternative (ii) gives (D1) with $m=n-3$. Since $P_Y$ is a $2$-group preserving the odd set $U$, it fixes some $r\in U$. Thus, restriction defines a homomorphism $P_Y\to S_{Y\setminus\{r\}}$ given by $\rho\mapsto\rho|_{Y\setminus\{r\}}$. This homomorphism is injective. After identifying $Y\setminus\{r\}$ with $\{1,\ldots,n-3\}$, denote the resulting subgroup of $S_{n-3}$ by $P_0$. Then $P_Y\cong P_0$ and $c(P_0)=c(P_Y)-1$. Since $P_Y$ has a fixed point, \eqref{eq:even-sign-rank} gives $E_Y(P_Y)\cong(\mathbb{Z}/2\mathbb{Z})^{c(P_Y)-1}=(\mathbb{Z}/2\mathbb{Z})^{c(P_0)}$. Therefore, alternative (ii), together with $H_*\cong\mathcal W_2$, gives (D1).

\emph{$(2)\Rightarrow(3)$.} We first derive (B0) from (D0). Let $\Omega_1,\ldots,\Omega_k$ be the orbits of $P$ on $\{1,\ldots,n\}$, with $|\Omega_j|=2^{a_j}$. The image of $P$ under restriction to each $\Omega_j$ is a $2$-subgroup of $S_{\Omega_j}$, and hence embeds into $\mathcal W_{a_j}$ by Sylow's theorem and Lemma~\ref{lem:sylow}. The homomorphism $P\to\prod_{j=1}^k S_{\Omega_j}$ defined by these restrictions is injective, since an element in its kernel fixes every element of $\{1,\ldots,n\}$. Consequently, $P\hookrightarrow\prod_{j=1}^k\mathcal W_{a_j}$.

If $P$ has no fixed point, then every $a_j\geq1$, so $n$ is even. By \eqref{eq:even-sign-rank},
\[
 E_D(P)\times P
 \hookrightarrow\prod_{j=1}^k B_{a_j},
 \qquad
 \sum_{j=1}^k2^{a_j}=n=2\lfloor n/2\rfloor.
\]
If $P$ has a fixed point, arrange that $\Omega_k$ is a singleton. As in the proof of $(1)\Rightarrow(2)$, the restriction homomorphism $P\to\prod_{j=1}^{k-1}S_{\Omega_j}$ is injective, since $\Omega_k$ is a singleton orbit. Moreover, \eqref{eq:even-sign-rank} gives $E_D(P)\cong(\mathbb{Z}/2\mathbb{Z})^{k-1}$. Consequently,
\[
 E_D(P)\times P
 \hookrightarrow\prod_{j=1}^{k-1}B_{a_j},
 \qquad
 \sum_{j=1}^{k-1}2^{a_j}
 =n-1\leq2\lfloor n/2\rfloor.
\]
This proves (B0).

We next derive (B1) from (D1). For $P\leq S_m$, let its orbits have sizes $2^{a_1},\ldots,2^{a_k}$, where $k=c(P)$. The same restriction argument gives
\[
 \mathcal W_2\times(\mathbb{Z}/2\mathbb{Z})^{c(P)}\times P
 \hookrightarrow\mathcal W_2\times\prod_{j=1}^kB_{a_j}.
\]
Since $\sum_j2^{a_j}=m\leq n-3$, we deduce that (B1) holds. When $m=0$, the product on the right is empty.

\emph{$(3)\Rightarrow(1)$.} We realize every group appearing in (B0) or (B1). For both constructions, we use the elements
\[
 \delta_0=\delta_1=1,\qquad
 \delta_2=\sigma_{\{s_1\}}\sigma_{\{s_2\}},\qquad
 \delta_r=\sigma_{\Delta_r}\quad(3\leq r\leq n).
\]
These elements commute and satisfy $\delta_r^2=1$. Indeed, opposition on each prefix preserves every smaller prefix. It also either fixes or exchanges $\{s_1\}$ and $\{s_2\}$, whose corresponding generators commute. By \eqref{eq:Dlongest},
\begin{equation}\label{eq:deltar}
 g_{W(D_n)}(\delta_r)=
 \begin{cases}
  1,&r=0,1,\\
  \mathrm d_{\{1,\ldots,r\}},&r\geq2\text{ even},\\
  \mathrm d_{\{2,\ldots,r\}},&r\geq3\text{ odd}.
 \end{cases}
\end{equation}

\emph{Realization of (B0).} We first reduce to nonempty products with $a_j\geq1$. Since $B_0\times B_0\cong B_1$, we may replace each pair of factors $B_0$ by one factor $B_1$. If one factor $B_0$ remains, embed the product into its product with one additional $B_0$, and replace the resulting pair by $B_1$. This increases $\sum_j2^{a_j}$ by one. The sum before this increase is odd, whereas $2\lfloor n/2\rfloor$ is even, so the required inequality is preserved. Therefore, it suffices to realize the resulting products with all $a_j\geq1$.

We now construct the factors $B_{a_j}=\mathbb{Z}/2\mathbb{Z}\times\mathcal W_{a_j}$. Let $M=\sum_j2^{a_j}$ and partition $\{1,\ldots,M\}$ into consecutive intervals $\Omega_j=\{p_j,\ldots,q_j\}$ of sizes $2^{a_j}$. Let $T=\{s_2,\ldots,s_n\}$, which is a subdiagram of type $A_{n-1}$. Using the identification $J_n\cong C_{W_T}$, Proposition~\ref{prop:parabolic} gives an embedding
\[
 \iota:J_n\hookrightarrow C_{W(D_n)},
 \qquad
 s_{r,s}\longmapsto\sigma_{A^+_{r,s}}.
\]
Apply Lemma~\ref{lem:dyadic} to each interval $\Omega_j=\{p_j,\ldots,q_j\}$ and let $L_j$ be the image under $\iota$ of the resulting subgroup. Then
\[
 L_j\leq
 \left\langle
  \sigma_{A^+_{r,s}}:p_j\leq r<s\leq q_j
 \right\rangle,
 \qquad
 L_j\cong\mathcal W_{a_j}.
\]
Regarding $S_n$ as a subgroup of $W(D_n)$, we have $g_{W(D_n)}\circ\iota=\Sigma_n$. Thus, $g_{W(D_n)}$ is injective on $L_j$ and $g_{W(D_n)}(L_j)\leq S_{\Omega_j}$. For the factor $\mathbb{Z}/2\mathbb{Z}$, let $z_j=\delta_{q_j}\delta_{p_j-1}$.  Both indices are even, so $z_j^2=1$ and $g_{W(D_n)}(z_j)=\mathrm d_{\Omega_j}$.

We verify that these choices define a homomorphism from $\prod_jB_{a_j}$. The groups $L_j$ commute with one another because their defining subdiagrams are pairwise orthogonal, and the elements $z_j$ commute because the $\delta_r$ commute. To check that $z_i$ commutes with $L_j$, consider a prefix $\Delta_b$ occurring in $z_i$. The integer $b$ is an endpoint of one of the intervals in the partition. Thus, each $A^+_{r,s}$ used to generate $L_j$ is contained in $\Delta_b$ or orthogonal to it. Since $b$ is even, opposition on $\Delta_b$ is trivial, so the defining relations give the required commutation. If $\delta_2$ occurs, the first interval is $\{1,2\}$. In that case, $\delta_2$ commutes with $\sigma_{\{s_2\}}$, and both $\{s_1\}$ and $\{s_2\}$ are orthogonal to every subdiagram used for the later intervals. Hence, $[z_i,L_j]=1$ for all $i,j$. Therefore, multiplication defines a homomorphism $\prod_jB_{a_j}\rightarrow C_{W(D_n)}$.

We show that this homomorphism is injective. Suppose that $\prod_j z_j^{\varepsilon_j}\ell_j=1$ for some $\varepsilon_j\in\{0,1\}$ and $\ell_j\in L_j$. For $i\in\Omega_j$, applying $g_{W(D_n)}$ gives $(-1)^{\varepsilon_j}g_{W(D_n)}(\ell_j)(i)=i$, because all factors belonging to other intervals fix every integer in $\Omega_j$. Since $g_{W(D_n)}(\ell_j)\in S_{\Omega_j}$, this forces $\varepsilon_j=0$ and $g_{W(D_n)}(\ell_j)=1$. Injectivity on $L_j$ then implies $\ell_j=1$. Thus, every product obtained after the initial reduction embeds into $C_{W(D_n)}$. Each original product is isomorphic to a subgroup of such a product, proving the realization of (B0).

\emph{Realization of (B1).} We first realize the factor $\mathcal W_2$ as $H^*= \left\langle \sigma_{\{s_1\}},\sigma_{\{s_2\}},\delta_3 \right\rangle$. The first two generators commute, and conjugation by $\delta_3$ exchanges them. These relations give a surjection $\mathcal W_2\to H^*$. Their images under $g_{W(D_n)}$ generate a dihedral group of order eight, so this surjection is an isomorphism and $g_{W(D_n)}$ is injective on $H^*$. If the product of the $B_{a_j}$ is empty, $H^*$ alone gives the required realization. Therefore, we may assume that the product is nonempty.

To construct the remaining factors commuting with $H^*$, define
\[
 \gamma_3=1,\qquad
 \gamma_r=
 \begin{cases}
  \delta_r,&r\geq4\text{ even},\\
  \delta_3\delta_r,&r\geq5\text{ odd}.
 \end{cases}
\]
These elements commute, satisfy $\gamma_r^2=1$, and commute with every element of $H^*$. For even $r$, opposition on $\Delta_r$ is trivial. For odd $r$, conjugation by each of $\delta_3$ and $\delta_r$ exchanges $\sigma_{\{s_1\}}$ and $\sigma_{\{s_2\}}$, so their product commutes with both generators. All the $\gamma_r$ also commute with $\delta_3$. By \eqref{eq:deltar},
\begin{equation}\label{eq:gammar}
 g_{W(D_n)}(\gamma_r)=
 \begin{cases}
  1,&r=3,\\
  \mathrm d_{\{1,\ldots,r\}},&r\geq4\text{ even},\\
  \mathrm d_{\{4,\ldots,r\}},&r\geq5\text{ odd}.
 \end{cases}
\end{equation}

We now construct each remaining factor $B_{a_j}$. Let $M=3+\sum_j2^{a_j}$ and partition $\{4,\ldots,M\}$ into consecutive intervals $\Omega_j=\{p_j,\ldots,q_j\}$ of sizes $2^{a_j}$. Choose $L_j\cong\mathcal W_{a_j}$ using the generators $\sigma_{A^+_{r,s}}$ with $p_j\leq r<s\leq q_j$, as in the realization of (B0). Let $z_j=\gamma_{q_j}\gamma_{p_j-1}$. Then $z_j^2=1$, and \eqref{eq:gammar} gives
\[
 g_{W(D_n)}(z_j)=
 \begin{cases}
  \mathrm d_{\Omega_j},&|\Omega_j|>1,\\
  \mathrm d_{\{1,2,3,p_j\}},&\Omega_j=\{p_j\}.
 \end{cases}
\]

We check that multiplication defines a homomorphism
\[
 H^*\times\prod_j(\mathbb{Z}/2\mathbb{Z}\times L_j)
 \longrightarrow C_{W(D_n)}.
\]
The groups $L_j$ commute with one another and with $H^*$: their defining subdiagrams are pairwise orthogonal and are also orthogonal to $\Delta_3$. The elements $z_j$ commute with one another and with $H^*$ since every $\gamma_r$ commutes with every element of $H^*$. Finally, each subdiagram $A^+_{r,s}$ used for an $L_j$ is contained in or orthogonal to every prefix $\Delta_b$ arising from an endpoint in the partition. Opposition on such a prefix fixes every contained $A^+_{r,s}$ pointwise. Furthermore, $\Delta_3$ is orthogonal to all these $A^+_{r,s}$, whose vertices belong to $\{s_5,\ldots,s_n\}$. Therefore, it follows that $[z_i,L_j]=1$ for all $i,j$.

Finally, we prove that this homomorphism is injective. Suppose that
\[
 h\prod_j z_j^{\varepsilon_j}\ell_j=1,
 \qquad
 h\in H^*,\quad
 \varepsilon_j\in\{0,1\},\quad
 \ell_j\in L_j.
\]
For a non-singleton interval $\Omega_j$, the image of $h$ and the images of all factors belonging to other intervals fix every integer in $\Omega_j$. The argument used for (B0) therefore gives $\varepsilon_j=0$ and $\ell_j=1$. For a singleton interval $\Omega_j=\{p_j\}$, we have $\ell_j=1$ since $L_j$ is trivial, and among the remaining factors only $g_{W(D_n)}(z_j)$ sends $p_j$ to $-p_j$. Thus, $\varepsilon_j=0$ in this case as well. The original relation now reduces to $h=1$. Consequently, we have $H^*\times\prod_j(\mathbb{Z}/2\mathbb{Z}\times L_j)\cong\mathcal W_2\times\prod_jB_{a_j}$. This realizes (B1) and completes the proof.
\end{proof}

In condition (3), the inequalities in (B0) and (B1) may be replaced by
\[
 \sum_j2^{a_j}=2\lfloor n/2\rfloor
 \quad\text{and}\quad
 \sum_j2^{a_j}=n-3,
\]
respectively. Indeed, adjoining a factor $B_0$ increases $\sum_j2^{a_j}$ by one, and the original group embeds into the enlarged group. Thus, every product satisfying either inequality embeds into a product satisfying the corresponding equality. Conversely, the products satisfying these equalities already satisfy the original inequalities. This gives the $D_n$ row of Table~\ref{tab:classification}.

\begin{remark}
For every $n\geq4$, no Sylow $2$-subgroup of $W(D_n)$ embeds into $C_{W(D_n)}$. Indeed, since $|B_a|=2^{2^a}$ and $|\mathcal W_2|=8$, Theorem~\ref{thm:typeD} implies that every finite subgroup of $C_{W(D_n)}$ has order at most $2^n$. The group $W(D_n)$ has order $2^{n-1}n!$, and $8$ divides $n!$ for $n\geq4$. Thus, a Sylow $2$-subgroup of $W(D_n)$ has order at least $2^{n+2}$ and cannot embed into $C_{W(D_n)}$.
\end{remark}

\section{\texorpdfstring{Classification in types $E_6$, $E_7$, and $E_8$}{Classification in types E\unichar{"2086}, E\unichar{"2087}, and E\unichar{"2088}}}\label{sec:typeE}
We number the $E_r$ diagram, $r\in\{6,7,8\}$, by taking the induced subdiagram on $\{s_1,\ldots,s_r\}$. Write $S_r=\{s_1,\ldots,s_r\}$. For a string of distinct indices, let $I_{i_1\cdots i_k}=\{s_{i_1},\ldots,s_{i_k}\}$. In type $E_6$, conjugation by $\omega_{S_6}$ exchanges $s_2$ with $s_6$ and $s_3$ with $s_5$, and fixes $s_1$ and $s_4$. In types $E_7$ and $E_8$, the longest elements $\omega_{S_7}$ and $\omega_{S_8}$ are central in their respective Coxeter groups; see \cite[Remark~3.1.2]{DJS03}.

\begin{figure}[htbp]
\centering
\begin{tikzpicture}[coxeter]

\begin{scope}
\node (b6) {$s_2$};
\node (c6) [right=of b6] {$s_3$};
\node (d6) [right=of c6] {$s_4$};
\node (e6) [right=of d6] {$s_5$};
\node (f6) [right=of e6] {$s_6$};
\node (a6) [above=of d6] {$s_1$};
\node [left=10mm of b6] {$E_6$};

\foreach \u/\v in {b6/c6,c6/d6,d6/e6,e6/f6}
  \draw (\u.east)--(\v.west);
\draw (a6.south)--(d6.north);
\end{scope}

\begin{scope}[yshift=-2cm]
\node (b7) {$s_2$};
\node (c7) [right=of b7] {$s_3$};
\node (d7) [right=of c7] {$s_4$};
\node (e7) [right=of d7] {$s_5$};
\node (f7) [right=of e7] {$s_6$};
\node (g7) [right=of f7] {$s_7$};
\node (a7) [above=of d7] {$s_1$};
\node [left=10mm of b7] {$E_7$};

\foreach \u/\v in {b7/c7,c7/d7,d7/e7,e7/f7,f7/g7}
  \draw (\u.east)--(\v.west);
\draw (a7.south)--(d7.north);
\end{scope}

\begin{scope}[yshift=-4cm]
\node (b8) {$s_2$};
\node (c8) [right=of b8] {$s_3$};
\node (d8) [right=of c8] {$s_4$};
\node (e8) [right=of d8] {$s_5$};
\node (f8) [right=of e8] {$s_6$};
\node (g8) [right=of f8] {$s_7$};
\node (h8) [right=of g8] {$s_8$};
\node (a8) [above=of d8] {$s_1$};
\node [left=10mm of b8] {$E_8$};

\foreach \u/\v in {b8/c8,c8/d8,d8/e8,e8/f8,f8/g8,g8/h8}
  \draw (\u.east)--(\v.west);
\draw (a8.south)--(d8.north);
\end{scope}

\end{tikzpicture}
\caption{Coxeter diagrams of types $E_6$, $E_7$, and $E_8$.}
\label{fig:typeE}
\end{figure}

\subsection{Embeddings of cube stabilizers}

We first collect the consequences of the results for types $A$ and $D$ that will be used for proper connected subdiagrams of $E_6$, $E_7$, and $E_8$.

\begin{lemma}\label{lem:lowrank}
If a finite group $F$ embeds into $C_{W(X)}$ for a type $X$ in the left column, it embeds into at least one group in the right column.
\[
\begin{array}{c|l}
 X&\text{groups containing an isomorphic copy of }F\\ \hline
 A_1,A_2&\mathbb{Z}/2\mathbb{Z}\\
 A_3,A_4&\mathcal W_2\\
 A_5,A_6&\mathcal W_2\times\mathbb{Z}/2\mathbb{Z}\\
 A_7&\mathcal W_3\\
 D_4&(\mathbb{Z}/2\mathbb{Z})^{4},\ \mathcal W_2\times\mathbb{Z}/2\mathbb{Z}\\
 D_5&(\mathbb{Z}/2\mathbb{Z})^{4},\ \mathcal W_2\times(\mathbb{Z}/2\mathbb{Z})^{2}\\
 D_6&(\mathbb{Z}/2\mathbb{Z})^{6},\ \mathcal W_2\times(\mathbb{Z}/2\mathbb{Z})^{3}\\
 D_7&(\mathbb{Z}/2\mathbb{Z})^{6},\ \mathcal W_2\times(\mathbb{Z}/2\mathbb{Z})^{4},\
       \mathcal W_2^2\times\mathbb{Z}/2\mathbb{Z}
\end{array}
\]
\end{lemma}

\begin{proof}
For $A_m$, apply Theorem~\ref{thm:typeA} using the binary expansion of $m+1$. For $D_m$, apply the two block-size bounds in Theorem~\ref{thm:typeD} for $4\leq m\leq7$. The only possible block sizes are $1$, $2$, and $4$, with
\[
 B_0=\mathbb{Z}/2\mathbb{Z},\qquad
 B_1\cong(\mathbb{Z}/2\mathbb{Z})^{2},\qquad
 B_2=\mathcal W_2\times\mathbb{Z}/2\mathbb{Z}.
\]
The resulting products embed into the groups listed in the lemma by enlarging elementary abelian factors and using $(\mathbb{Z}/2\mathbb{Z})^{2}\hookrightarrow\mathcal W_2$. In particular, the group $\mathcal W_2\times(\mathbb{Z}/2\mathbb{Z})^{3}$ arising from (B0) for $D_7$ embeds into $\mathcal W_2\times(\mathbb{Z}/2\mathbb{Z})^{4}$, which is listed for type $D_7$.
\end{proof}

For $E_6$, we will also need to include a generator whose conjugation action is given by opposition. The following lemma identifies the resulting group.

\begin{lemma}\label{lem:oppositionextension}
Let $R$ be a connected spherical subset of the generating set of a Coxeter system. Let $\alpha$ be opposition on $R$, and suppose that $N\subsetneq R$ satisfies $\alpha(N)=N$. The automorphism of $C_N$ induced by $\alpha$ is given by
\[
 \sigma_I\longmapsto\sigma_{\alpha(I)}
 \qquad(I\in\mathcal I(N)).
\]
The assignments
\[
 \Phi:C_N\rtimes_\alpha\langle t\rangle
       \longrightarrow C_R,
 \qquad
 \Phi(x)=x\quad(x\in C_N),\qquad
 \Phi(t)=\sigma_R,\qquad t^2=1,
\]
define an embedding with image $\langle C_N,\sigma_R\rangle$.
\end{lemma}

\begin{proof}
By Proposition~\ref{prop:parabolic}, we regard $C_N$ as a subgroup of $C_R$. The nesting relations give
\[
 \sigma_R\sigma_I\sigma_R
 =\sigma_{\alpha(I)}
 \qquad(I\in\mathcal I(N)).
\]
Thus, conjugation by $\sigma_R$ induces the required automorphism of $C_N$. Together with $\sigma_R^2=1$, this shows that $\Phi$ is a homomorphism onto $\langle C_N,\sigma_R\rangle$.

To prove injectivity, it remains to check elements of the form $xt$, with $x\in C_N$. We have
\[
 g_{W_R}(\Phi(xt))
 =g_{W_R}(x)\omega_R\in W_N\omega_R.
\]
Since $N\subsetneq R$, the longest element $\omega_R$ does not belong to $W_N$, so this product is not $1$. Hence, $\Phi(xt)\ne1$, proving injectivity.
\end{proof}

We now consider type $E_6$. By Theorem~\ref{thm:criterion}, the following lemma shows that every finite subgroup of $C_{W(E_6)}$ embeds into $\mathcal W_2\times(\mathbb{Z}/2\mathbb{Z})^{2}$.

\begin{lemma}\label{lem:E6}
Every cube stabilizer for the action of $C_{W(E_6)}$ on $\mathscr A_{W(E_6)}$ embeds into $\mathcal W_2\times(\mathbb{Z}/2\mathbb{Z})^{2}$.
\end{lemma}

\begin{proof}
After translating the cube, it suffices to consider $K_{\mathcal A}$ for a compatible family $\mathcal A$.

We first consider the case $S_6\notin\mathcal A$ by decomposing $K_{\mathcal A}$ over the maximal members $M_1,\ldots,M_k$ of $\mathcal A$. Let $\mathcal A_i=\{I\in\mathcal A:I\subseteq M_i\}$ and let $K_i=\operatorname{Stab}_{C_{M_i}}(Q_{\mathcal A_i})$. Lemma~\ref{lem:maxsupport} gives $K_{\mathcal A}\cong\prod_iK_i$. Moreover, we have $\bigcup_iM_i\subsetneq S_6$ by Lemma~\ref{lem:dimension}, applied to the compatible family $\mathcal A\cup\{S_6\}$. We distinguish the possibilities according to the largest cardinality of an $M_i$.

If some $M_i$ has five vertices, it is the only maximal member. Its type is $A_5$ or $D_5$, and every corresponding group in Lemma~\ref{lem:lowrank} embeds into $\mathcal W_2\times(\mathbb{Z}/2\mathbb{Z})^{2}$.

Suppose that the largest cardinality is four. Such an $M_i$ has type $A_4$ or $D_4$. For $A_4$, Lemma~\ref{lem:lowrank} gives $K_i\hookrightarrow\mathcal W_2$. The other maximal members occupy at most one vertex, so their contribution gives $K_{\mathcal A}\hookrightarrow\mathcal W_2\times\mathbb{Z}/2\mathbb{Z}$. The only subdiagram of type $D_4$ is $I_{1345}$. Both remaining vertices are adjacent to it, so it is the only maximal member. Its two bounds in Lemma~\ref{lem:lowrank} both embed into $\mathcal W_2\times(\mathbb{Z}/2\mathbb{Z})^{2}$.

If the largest cardinality is three, such an $M_i$ has type $A_3$ and $K_i\hookrightarrow\mathcal W_2$. The other maximal members occupy at most two vertices. Each of their stabilizers has order at most two, so together they contribute at most two factors $\mathbb{Z}/2\mathbb{Z}$. Again, $K_{\mathcal A}\hookrightarrow\mathcal W_2\times(\mathbb{Z}/2\mathbb{Z})^{2}$.

Finally, suppose that every $M_i$ has at most two vertices. Each $K_i$ then has order at most two. To bound the number of these factors, choose one vertex from each $M_i$. No two chosen vertices are adjacent, since the $M_i$ are pairwise orthogonal. Partition $S_6$ into the three pairs $\{s_1,s_4\}$, $\{s_2,s_3\}$, and $\{s_5,s_6\}$. The two vertices in each pair are adjacent. Since no two chosen vertices are adjacent, each pair contains at most one chosen vertex. Thus $k\leq3$, and
\[
 K_{\mathcal A}\hookrightarrow(\mathbb{Z}/2\mathbb{Z})^{3}
 \hookrightarrow\mathcal W_2\times(\mathbb{Z}/2\mathbb{Z})^{2}.
\]

We now consider the case $S_6\in\mathcal A$. Let $\mathcal D=\mathcal A\setminus\{S_6\}$, let $M_1,\ldots,M_k$ be the maximal members of $\mathcal D$, and put $N=\bigcup_iM_i$. If $\mathcal D=\varnothing$, then $K_{\mathcal A}=\langle\sigma_{S_6}\rangle\cong\mathbb{Z}/2\mathbb{Z}$, so assume that $\mathcal D\ne\varnothing$.

We first separate the elements represented by vertices whose expressions do not contain $\sigma_{S_6}$. For $\mathcal U\subseteq\mathcal D$, Theorem~\ref{thm:criterion} gives
\[
 v_{\mathcal U}\in K_{\mathcal A}
 \quad\Longleftrightarrow\quad
 v_{\mathcal U}\in K_{\mathcal D}.
\]
Indeed, we have $\mathcal P_{\mathcal A}=\mathcal P_{\mathcal D}\cup\{W(E_6)\}$. Conjugation by $w_{\mathcal U}$ fixes $W(E_6)$, so it permutes the subgroups in $\mathcal P_{\mathcal A}$ if and only if it permutes those in $\mathcal P_{\mathcal D}$. If no element $v_{\mathcal U}\sigma_{S_6}$ belongs to $K_{\mathcal A}$, then $K_{\mathcal A}=K_{\mathcal D}$, and the preceding case applies.

Otherwise, we show that the existence of such an element forces $N$ to be preserved by opposition on $S_6$. Write
\[
 \theta=(s_2\ s_6)(s_3\ s_5),
 \qquad
 \theta(s_1)=s_1,\quad\theta(s_4)=s_4,
\]
and choose $v_{\mathcal U}\sigma_{S_6}\in K_{\mathcal A}$. By Theorem~\ref{thm:criterion}, $w_{\mathcal U}\omega_{S_6}$ permutes $\{W_I:I\in\mathcal A\}$. It fixes $W(E_6)$, so it permutes the subgroups in $\{W_I:I\in\mathcal D\}$. Since conjugation preserves inclusion of subgroups, it also permutes the maximal members of this collection, namely $W_{M_1},\ldots,W_{M_k}$. Consequently, for every $i$ there is a $j$ such that $w_{\mathcal U}W_{\theta(M_i)} w_{\mathcal U}^{-1}=W_{M_j}$. Every factor in the product defining $w_{\mathcal U}$ belongs to $W_N$, so $w_{\mathcal U}\in W_N$. Hence,
\[
 W_{\theta(M_i)}
 =w_{\mathcal U}^{-1}W_{M_j}w_{\mathcal U}
 \leq W_N.
\]
Containment of standard parabolic subgroups gives $\theta(M_i)\subseteq N$. Taking unions yields $\theta(N)\subseteq N$, and applying the involution $\theta$ gives $\theta(N)=N$. Moreover, $N\subsetneq S_6$ by Lemma~\ref{lem:dimension}.

Every vertex of $Q_{\mathcal A}$ has the form $v_{\mathcal U}$ or $v_{\mathcal U}\sigma_{S_6}$, with $\mathcal U\subseteq\mathcal D$. Since $v_{\mathcal U}\in C_N$ for every $\mathcal U\subseteq\mathcal D$, all vertices of $Q_{\mathcal A}$ lie in $\langle C_N,\sigma_{S_6}\rangle$. Therefore, Lemma~\ref{lem:oppositionextension} gives
\[
 K_{\mathcal A}
 \leq\langle C_N,\sigma_{S_6}\rangle
 \cong C_N\rtimes_\theta\mathbb{Z}/2\mathbb{Z}.
\]

To bound the finite subgroups of this extension, we enumerate the possible sets $N$. The four $\theta$-orbits in $S_6$ give fourteen nonempty proper invariant subsets. The connected ones are
\begin{equation}\label{eq:E6connected}
 I_1,\quad I_4,\quad I_{14},\quad
 I_{345},\quad I_{1345},\quad I_{23456}.
\end{equation}
The disconnected ones are
\begin{align*}
 &I_2\sqcup I_6,\quad
 I_3\sqcup I_5,\quad
 I_{23}\sqcup I_{56},\\
 &I_1\sqcup I_2\sqcup I_6,\quad
 I_1\sqcup I_3\sqcup I_5,\quad
 I_1\sqcup I_{23}\sqcup I_{56},\\
 &I_2\sqcup I_4\sqcup I_6,\quad
 I_{14}\sqcup I_2\sqcup I_6.
\end{align*}

For $N=I_1$, $I_4$, or $I_{14}$, $\theta$ fixes $N$ pointwise, so $\langle C_N,\sigma_{S_6}\rangle \cong C_N\times\mathbb{Z}/2\mathbb{Z}$. Let $H$ be a finite subgroup of this direct product, and let $H_N$ be its projection onto $C_N$. By Lemma~\ref{lem:lowrank}, $H_N\hookrightarrow\mathbb{Z}/2\mathbb{Z}$. Hence, $H\leq H_N\times\mathbb{Z}/2\mathbb{Z}\hookrightarrow(\mathbb{Z}/2\mathbb{Z})^{2}$.

For $N=I_{345}$ or $I_{23456}$, we also reduce the extension to a direct product. These subdiagrams have type $A_3$ and $A_5$, respectively, and $\theta|_N$ is opposition on $N$. Let $a=\sigma_N$ and $t=\sigma_{S_6}$. The involutions $a$ and $t$ commute and induce the same automorphism of $C_N$. Thus, $z=at$ is a central involution. By Lemma~\ref{lem:oppositionextension}, $C_N\cap\langle t\rangle=\{1\}$, so $t\notin C_N$. Since $a\in C_N$ and $t=az$, we also have $z\notin C_N$. Therefore, we obtain $\langle C_N,\sigma_{S_6}\rangle\cong C_N\times\mathbb{Z}/2\mathbb{Z}$. Let $H$ be a finite subgroup of $C_N\times\mathbb{Z}/2\mathbb{Z}$, and let $H_N$ be its projection onto $C_N$. Then $H_N$ is finite and $H\leq H_N\times\mathbb{Z}/2\mathbb{Z}$. By Lemma~\ref{lem:lowrank}, $H_N$ embeds into $\mathcal W_2$ when $N=I_{345}$ and into $\mathcal W_2\times\mathbb{Z}/2\mathbb{Z}$ when $N=I_{23456}$. Consequently, $H$ embeds into $\mathcal W_2\times\mathbb{Z}/2\mathbb{Z}$ or $\mathcal W_2\times(\mathbb{Z}/2\mathbb{Z})^{2}$, respectively.

For $N=I_{1345}$, we embed the extension into the generalized cactus group of type $D_5$. Attach a new vertex only to $s_1$ to obtain a diagram $\widehat N$ of type $D_5$. Opposition on $\widehat N$ exchanges $s_3$ and $s_5$ and fixes $s_1$ and $s_4$, so its restriction to $N$ is $\theta|_N$. Thus, Lemma~\ref{lem:oppositionextension} gives an embedding
\[
 C_N\rtimes_\theta\langle t\rangle
 \hookrightarrow C_{W(\widehat N)},
 \qquad
 \sigma_I\longmapsto\sigma_I,\quad
 t\longmapsto\sigma_{\widehat N}.
\]
Both groups in the $D_5$ row of Lemma~\ref{lem:lowrank} embed into $\mathcal W_2\times(\mathbb{Z}/2\mathbb{Z})^{2}$, using $(\mathbb{Z}/2\mathbb{Z})^{2}\hookrightarrow\mathcal W_2$.

Suppose that $N$ is disconnected. In each case, write
\[
 N=M\sqcup\theta(M)\sqcup L,
\]
where $M$ and $\theta(M)$ are components of type $A_m$, with $m=1$ or $2$, and $L$ is either empty or a component of type $A_1$ or $A_2$ fixed pointwise by $\theta$. We use the convention $C_\varnothing=\{1\}$. We will embed $(C_M\times C_M)\rtimes\mathbb{Z}/2\mathbb{Z}$ into $C_{W(A_{2m+1})}$.

Identifying $\theta(M)$ with $M$ by $\theta$, conjugation by $\sigma_{S_6}$ exchanges the two copies of $C_M$ and acts trivially on $C_L$. Hence, 
\[
 \langle C_N,\sigma_{S_6}\rangle
 \cong
 \bigl((C_M\times C_M)\rtimes\mathbb{Z}/2\mathbb{Z}\bigr)\times C_L.
\]

To embed the first factor, take a diagram of type $A_{2m+1}$ with generating set $R=\{t_1,\ldots,t_{2m+1}\}$ and edges $\{t_i,t_{i+1}\}$. The subsets $R_-=\{t_1,\ldots,t_m\}$ and $R_+=\{t_{m+2},\ldots,t_{2m+1}\}$ are orthogonal and exchanged by opposition $\rho(t_i)=t_{2m+2-i}$. Choose a diagram isomorphism from $M$ to $R_-$, inducing an isomorphism $\iota_-:C_M\to C_{R_-}$, and define $\iota_+(x)=\sigma_R\iota_-(x)\sigma_R \in C_{R_+}$. The two images commute, and conjugation by $\sigma_R$ exchanges them. Let $N'=R_-\cup R_+$. Then $N'=R\setminus\{t_{m+1}\}$ is a proper subset of $R$ and satisfies $\rho(N')=N'$. Moreover, $R_-\perp R_+$ gives $C_{N'}\cong C_{R_-}\times C_{R_+}$. Applying Lemma~\ref{lem:oppositionextension} to $N'\subsetneq R$ and using the identifications $\iota_-$ and $\iota_+$ gives the embedding
\[
 (C_M\times C_M)\rtimes\mathbb{Z}/2\mathbb{Z}\hookrightarrow C_R,
 \qquad
 ((x,y),t^\varepsilon)
 \longmapsto
 \iota_-(x)\iota_+(y)\sigma_R^\varepsilon.
\]

Let $H$ be a finite subgroup of $\bigl((C_M\times C_M)\rtimes\mathbb{Z}/2\mathbb{Z}\bigr)\times C_L$, and let $H_1$ and $H_L$ be its projections onto the two factors. By Lemma~\ref{lem:lowrank}, $H_1$ embeds into $\mathcal W_2$ for $m=1$ and into $\mathcal W_2\times\mathbb{Z}/2\mathbb{Z}$ for $m=2$. Since $L$ is empty or has type $A_1$ or $A_2$, we also have $H_L\hookrightarrow\mathbb{Z}/2\mathbb{Z}$. Therefore,
\[
 H\leq H_1\times H_L
 \hookrightarrow\mathcal W_2\times(\mathbb{Z}/2\mathbb{Z})^{2}.
\]
\end{proof}

The following lemma combines the preceding results for types $A$, $D$, and $E_6$ in a form that we will use to analyze cube stabilizers in types $E_7$ and $E_8$.

\begin{lemma}\label{lem:rank6}
Let $M$ have type $A$, $D$, or $E_6$, and put $m=|M|\leq6$. Every finite subgroup of $C_M$ embeds into at least one of
\[
 (\mathbb{Z}/2\mathbb{Z})^{m},\qquad
 \mathcal W_2\times(\mathbb{Z}/2\mathbb{Z})^{m-3},
\]
where the second group is used only for $m\geq3$.
\end{lemma}

\begin{proof}
For types $A$ and $D$, the assertion follows from Lemma~\ref{lem:lowrank}. For type $E_6$, Theorem~\ref{thm:criterion} places the finite subgroup in a cube stabilizer. Lemma~\ref{lem:E6} then gives an embedding into $\mathcal W_2\times(\mathbb{Z}/2\mathbb{Z})^{2}$, which embeds into $\mathcal W_2\times(\mathbb{Z}/2\mathbb{Z})^{3}$.
\end{proof}

We are now ready to analyze the cube stabilizers for the actions of $C_{W(E_7)}$ and $C_{W(E_8)}$.

\begin{lemma}\label{lem:E78}
Every cube stabilizer in $C_{W(E_7)}$ embeds into one of
\[
 (\mathbb{Z}/2\mathbb{Z})^{7},\qquad
 \mathcal W_2\times(\mathbb{Z}/2\mathbb{Z})^{4}.
\]
Every cube stabilizer in $C_{W(E_8)}$ embeds into one of
\[
 (\mathbb{Z}/2\mathbb{Z})^{8},\qquad
 \mathcal W_2\times(\mathbb{Z}/2\mathbb{Z})^{5},\qquad
 \mathcal W_2^2\times(\mathbb{Z}/2\mathbb{Z})^{2},\qquad
 \mathcal W_3\times\mathbb{Z}/2\mathbb{Z}.
\]
\end{lemma}

\begin{proof}
We first decompose the stabilizer using proper connected subdiagrams. As in Lemma~\ref{lem:E6}, it suffices to consider $K_{\mathcal A}$. Let $r\in\{7,8\}$, $\mathcal D=\mathcal A\setminus\{S_r\}$, and let $\delta=1$ if $S_r\in\mathcal A$ and $\delta=0$ otherwise. Let $M_1,\ldots,M_k$ be the maximal members of $\mathcal D$, and put
\[
 m_i=|M_i|,\qquad
 \mathcal D_i=\{I\in\mathcal D:I\subseteq M_i\},
 \qquad
 K_i=\operatorname{Stab}_{C_{M_i}}(Q_{\mathcal D_i}).
\]
Since $\omega_{S_r}$ is central, Lemmas~\ref{lem:central} and~\ref{lem:maxsupport} give $K_{\mathcal A} \cong\prod_iK_i\times(\mathbb{Z}/2\mathbb{Z})^{\delta}$. Note that the $M_i$ are pairwise disjoint. If $\delta=1$, their union is a proper subset of $S_r$ by Lemma~\ref{lem:dimension}. Consequently, $\sum_i m_i+\delta\leq r$.

Suppose first that every $m_i\leq6$. We count the factors for which the second group in Lemma~\ref{lem:rank6} is used. Every $M_i$ has type $A$, $D$, or $E_6$, so that lemma applies. Choose one of its two groups for each $K_i$, let $J$ be the indices for which the second group is chosen, and let
\[
 b=\delta+\sum_{i\notin J}m_i
          +\sum_{i\in J}(m_i-3).
\]
It follows that
\[
 K_{\mathcal A}
 \hookrightarrow\mathcal W_2^{|J|}\times(\mathbb{Z}/2\mathbb{Z})^{b},
 \qquad
 3|J|+b=\delta+\sum_i m_i\leq r.
\]

For type $E_7$, we show that $|J|\leq1$. Each $i\in J$ satisfies $m_i\geq3$, and the $M_i$ are pairwise orthogonal. Thus, it suffices to show that two orthogonal connected subsets of $S_7$ cannot both have at least three vertices.

Deleting $s_4$ leaves the components $\{s_1\}$, $\{s_2,s_3\}$, and $\{s_5,s_6,s_7\}$. If neither subset contains $s_4$, only the last component can contain a connected subset with at least three vertices, and it cannot contain two disjoint such subsets. If one subset contains $s_4$, every orthogonal subset avoids $s_4$ and its neighbors $s_1,s_3,s_5$. Thus, it lies in $\{s_2\}$ or $\{s_6,s_7\}$. Therefore, $|J|\leq1$. Together with $3|J|+b\leq7$, this gives
\[
 K_{\mathcal A}\hookrightarrow(\mathbb{Z}/2\mathbb{Z})^{7}
 \quad\text{or}\quad
 K_{\mathcal A}\hookrightarrow\mathcal W_2\times(\mathbb{Z}/2\mathbb{Z})^{4}.
\]

For type $E_8$, still assuming that every $m_i\leq6$, the inequality $3|J|+b\leq8$ gives $|J|\leq2$. If $|J|=0$, $1$, or $2$, then $b\leq8$, $5$, or $2$, respectively. Hence, $K_{\mathcal A}$ embeds into one of
\[
 (\mathbb{Z}/2\mathbb{Z})^{8},\qquad
 \mathcal W_2\times(\mathbb{Z}/2\mathbb{Z})^{5},\qquad
 \mathcal W_2^2\times(\mathbb{Z}/2\mathbb{Z})^{2}.
\]

It remains to consider $E_8$ when some $M_i$ has seven vertices. Such an $M_i$ is the only maximal member, because the remaining vertex is adjacent to it. The possibilities are
\[
 I_{2345678}\ (A_7),\qquad
 I_{1345678}\ (D_7),\qquad
 I_{1234567}\ (E_7).
\]
For $A_7$, Lemma~\ref{lem:lowrank} gives $K_i\hookrightarrow\mathcal W_3$, and adjoining the possible factor $\mathbb{Z}/2\mathbb{Z}$ gives $\mathcal W_3\times\mathbb{Z}/2\mathbb{Z}$. For $D_7$, the bounds in Lemma~\ref{lem:lowrank}, with at most one additional factor $\mathbb{Z}/2\mathbb{Z}$, embed into
\[
 (\mathbb{Z}/2\mathbb{Z})^{8},\qquad
 \mathcal W_2\times(\mathbb{Z}/2\mathbb{Z})^{5},\qquad
 \mathcal W_2^2\times(\mathbb{Z}/2\mathbb{Z})^{2}.
\]
For $E_7$, the bounds just proved, with at most one additional factor $\mathbb{Z}/2\mathbb{Z}$, embed into the first two of these groups.
\end{proof}

\subsection{Construction of cube stabilizers}

\begin{theorem}\label{thm:typeE}
For $r=6,7,8$, a finite group embeds into $C_{W(E_r)}$ if and only if it embeds into one of the groups in the $E_r$ row of Table~\ref{tab:classification}.
\end{theorem}

\begin{proof}
By Theorem~\ref{thm:criterion}, every finite subgroup embeds into some $K_{\mathcal A}$. Therefore, Lemmas~\ref{lem:E6} and~\ref{lem:E78} give the necessary conditions. We prove the converse by realizing each listed group as a subgroup of the corresponding generalized cactus group. For these constructions, we use the following observation. If a compatible family $\mathcal A$ satisfies $\sigma_I\in K_{\mathcal A}$ for every $I\in\mathcal A$, then
\begin{equation}\label{eq:fullcubeorder}
 K_{\mathcal A}
 =\langle\sigma_I:I\in\mathcal A\rangle,
 \qquad
 |K_{\mathcal A}|=2^{|\mathcal A|}.
\end{equation}
Indeed, the assumption gives $\langle\sigma_I:I\in\mathcal A\rangle \leq K_{\mathcal A}$. Conversely, since $K_{\mathcal{A}}\subseteq V(Q_{\mathcal{A}})$, and since each vertex in $V(Q_{\mathcal{A}})$ is a product of generators $\sigma_I$ with $I\in\mathcal A$ (Lemma~\ref{lem:vertices}), we have $K_{\mathcal{A}}\leq \langle\sigma_I:I\in\mathcal A\rangle$. Moreover, every vertex of $Q_{\mathcal A}$ belongs to this group, and Lemma~\ref{lem:vertices} gives $|V(Q_{\mathcal A})|=2^{|\mathcal A|}$. Hence, $|K_{\mathcal A}|=2^{|\mathcal A|}$.

To verify the assumption, fix $I\in\mathcal A$. By Theorem~\ref{thm:criterion}, we must check that $\omega_I\mathcal P_{\mathcal A}\omega_I^{-1}=\mathcal P_{\mathcal A}.$ For $J\in\mathcal A$ with $I\subseteq J$, we have $\omega_I\in W_J$, so $\omega_IW_J\omega_I^{-1}=W_J$. If $I\perp J$, then $\omega_I$ commutes with every element of $W_J$, giving the same equality. By compatibility, the remaining members satisfy $J\subsetneq I$. For these members, we have $\omega_IW_J\omega_I^{-1}=W_{\omega_I(J)}.$ Therefore, it suffices to check that
\[
 \{\omega_I(J):J\in\mathcal A,\ J\subsetneq I\}
 =
 \{J\in\mathcal A:J\subsetneq I\}.
\]

We realize $\mathcal W_2\times(\mathbb{Z}/2\mathbb{Z})^{2}$, $\mathcal W_2\times(\mathbb{Z}/2\mathbb{Z})^{4}$, and $\mathcal W_2\times(\mathbb{Z}/2\mathbb{Z})^{5}$ in types $E_6$, $E_7$, and $E_8$, respectively, by successively enlarging a compatible family. Define
\[
 \mathcal D_6=\{I_1,I_2,I_5,I_{145},I_{12345}\},\quad
 \mathcal D_7=\mathcal D_6\cup\{I_7,I_{1234567}\},\quad
 \mathcal D_8=\mathcal D_7\cup\{I_{12345678}\}.
\]
The chain $I_{145}\subset I_{12345}$ contains $I_1,I_5$, while $I_2\perp I_{145}$ and $I_7\perp I_{12345}$. The added full diagrams contain all preceding members, so all three families are compatible. The opposition actions in Table~\ref{tab:opposition} preserve these families. Thus \eqref{eq:fullcubeorder} applies.

To identify these stabilizers, let $x=\sigma_{I_1}$, $y=\sigma_{I_5}$, $a=\sigma_{I_{145}}$, and $q=\sigma_{I_{12345}}$. The construction in Lemma~\ref{lem:dyadic}, applied to $I_{145}$ using Proposition~\ref{prop:parabolic}, gives $\langle x,y,a\rangle\cong\mathcal W_2$. Both $a$ and $q$ exchange $x,y$ by conjugation, and $a,q$ commute. Hence, $z=aq$ is a central involution. For each $r=6,7,8$, the generators other than $x,y,a,q$ in the family $\mathcal D_r$ commute with every generator of $K_{\mathcal D_r}$, so they are central in $K_{\mathcal D_r}$. These relations give surjections onto the three stabilizers from
\[
 \mathcal W_2\times(\mathbb{Z}/2\mathbb{Z})^{2},\qquad
 \mathcal W_2\times(\mathbb{Z}/2\mathbb{Z})^{4},\qquad
 \mathcal W_2\times(\mathbb{Z}/2\mathbb{Z})^{5},
\]
respectively. Their orders are $2^5$, $2^7$, and $2^8$, equal to the stabilizer orders in \eqref{eq:fullcubeorder}. Therefore, the surjections are isomorphisms.

To realize $(\mathbb{Z}/2\mathbb{Z})^{7}$ and $(\mathbb{Z}/2\mathbb{Z})^{8}$ in types $E_7$ and $E_8$, take
\begin{align*}
 \mathcal E_7
 &=\{I_1,I_3,I_5,I_7,I_{1345},I_{134567},
       I_{1234567}\},\\
 \mathcal E_8
 &=\mathcal E_7\cup\{I_{12345678}\}.
\end{align*}
The singleton members are pairwise orthogonal. Moreover, $I_1,I_3,I_5\subset I_{1345}$, $I_7\perp I_{1345}$, and the larger members form a containing chain. Thus, the families are compatible. Their nonsingleton members have types $D_4,D_6,E_7,E_8$, whose longest elements are central. The corresponding generators therefore commute. Equation~\eqref{eq:fullcubeorder} identifies the stabilizers with $(\mathbb{Z}/2\mathbb{Z})^{7}$ and $(\mathbb{Z}/2\mathbb{Z})^{8}$.

To realize $\mathcal W_2^2\times(\mathbb{Z}/2\mathbb{Z})^{2}$ in type $E_8$, let
\[
 \mathcal B=
 \{I_1,I_3,I_6,I_8,I_{134},I_{678},
   I_{1345678},I_{12345678}\}.
\]
The subdiagrams $I_{134}$ and $I_{678}$ have type $A_3$ and are orthogonal. Each contains its two listed singleton members, and both are contained in $I_{1345678}$. Thus, $\mathcal B$ is compatible. Lemma~\ref{lem:dyadic} and Proposition~\ref{prop:parabolic} give two commuting subgroups isomorphic to $\mathcal W_2$.

Let $a=\sigma_{I_{134}}$, $d=\sigma_{I_{1345678}}$, and $e=\sigma_{I_{12345678}}$. Opposition on the $D_7$ subdiagram $I_{1345678}$ exchanges $s_1,s_3$ and fixes $I_{678}$ pointwise. Hence, $a$ and $d$ induce the same conjugation on both subgroups isomorphic to $\mathcal W_2$. They also commute, so $c=ad$ is a central involution. The element $e$ is central as well. All opposition actions preserve $\mathcal B$, as recorded in Table~\ref{tab:opposition}. Since $d=ac$, the preceding elements generate $K_{\mathcal B}$ and give a surjection $\mathcal W_2^2\times(\mathbb{Z}/2\mathbb{Z})^{2}\longrightarrow K_{\mathcal B}$. Both groups have order $2^8$ by \eqref{eq:fullcubeorder}, so this is an isomorphism.

To realize $\mathcal W_3\times\mathbb{Z}/2\mathbb{Z}$ in type $E_8$, let
\[
 \mathcal T=
 \{I_2,I_4,I_6,I_8,I_{234},I_{678},
   I_{2345678},I_{12345678}\}.
\]
The first seven members are the subdiagrams used in the dyadic construction inside the $A_7$ subdiagram $I_{2345678}$. Lemma~\ref{lem:dyadic} and Proposition~\ref{prop:parabolic} show that their generators generate a subgroup isomorphic to $\mathcal W_3$. The remaining generator $\sigma_{I_{12345678}}$ is central. The family is compatible, and the opposition actions in Table~\ref{tab:opposition} preserve it. Consequently, \eqref{eq:fullcubeorder} gives $|K_{\mathcal T}|=2^8$. The resulting surjection $\mathcal W_3\times\mathbb{Z}/2\mathbb{Z}\longrightarrow K_{\mathcal T}$ is therefore an isomorphism, since $|\mathcal W_3\times\mathbb{Z}/2\mathbb{Z}|=2^8=|K_{\mathcal T}|$.
\end{proof}

Table~\ref{tab:opposition} records the opposition actions used to verify \eqref{eq:fullcubeorder} in these constructions. For each member not shown in the second column, opposition fixes every member of the corresponding family contained in it.

\begin{table}[htbp]
\centering\small
\renewcommand{\arraystretch}{1.2}
\begin{tabular}{@{}lll@{}}
\toprule
Family & Members inducing opposition
       & Nontrivial permutation of members\\
\midrule
$\mathcal D_6,\mathcal D_7,\mathcal D_8$
 & $I_{145},\ I_{12345}$ & $(I_1\ I_5)$\\
$\mathcal E_7,\mathcal E_8$
 & none & none\\
$\mathcal B$
 & $I_{134},\ I_{1345678}$ & $(I_1\ I_3)$\\
$\mathcal B$
 & $I_{678}$ & $(I_6\ I_8)$\\
$\mathcal T$
 & $I_{234}$ & $(I_2\ I_4)$\\
$\mathcal T$
 & $I_{678}$ & $(I_6\ I_8)$\\
$\mathcal T$
 & $I_{2345678}$
 & $(I_2\ I_8)(I_4\ I_6)(I_{234}\ I_{678})$\\
\bottomrule
\end{tabular}
\par\vspace{6pt}
\caption{Opposition on the compatible families used in the exceptional realizations.}
\label{tab:opposition}
\end{table}

\begin{remark}\label{rem:Eirredundant}
None of the alternatives for $E_7$ or $E_8$ can be omitted. The two groups for $E_7$ both have order $128$ but have different exponents. The four groups for $E_8$ all have order $256$ and are distinguished by the orders of their commutator subgroups.
\[
\begin{array}{c|c}
 G&|[G,G]|\\ \hline
 (\mathbb{Z}/2\mathbb{Z})^{8}&1\\
 \mathcal W_2\times(\mathbb{Z}/2\mathbb{Z})^{5}&2\\
 \mathcal W_2^2\times(\mathbb{Z}/2\mathbb{Z})^{2}&4\\
 \mathcal W_3\times\mathbb{Z}/2\mathbb{Z}&16
\end{array}
\]
\end{remark}

\section{\texorpdfstring{Classification in types $F$, $H$, and $I_2$}{Classification in types F, H, and I\unichar{"2082}}}\label{sec:small}

We complete the classification for finite irreducible Coxeter systems by treating types $F_4$, $H_3$, $H_4$, and $I_2(p)$. We begin with $I_2(p)$ because the classification for $I_2(5)$ is used in the proofs for $H_3$ and $H_4$.

\subsection{\texorpdfstring{Type $I_2(p)$}{Type I\unichar{"2082}(p)}}

In type $I_2(p)$, we determine $C_{W(I_2(p))}$ directly from its defining presentation.

\begin{figure}[htbp]
\centering
\begin{tikzpicture}[coxeter]
\node (a) {$s_1$};
\node (b) [right=of a] {$s_2$};

\draw (a.east)--node[above] {$p$} (b.west);
\end{tikzpicture}
\caption{Coxeter diagram of type $I_2(p)$.}
\label{fig:typeI}
\end{figure}

\begin{proposition}\label{prop:dihedral}
For $p\geq3$,
\[
 C_{W(I_2(p))}\cong\begin{cases}
 D_\infty,&p\text{ odd},\\
 D_\infty\times\mathbb{Z}/2\mathbb{Z},&p\text{ even},
 \end{cases}
\qquad D_\infty=(\mathbb{Z}/2\mathbb{Z})*(\mathbb{Z}/2\mathbb{Z}).
\]
Consequently, every finite subgroup of $C_{W(I_2(p))}$ embeds into $\mathbb{Z}/2\mathbb{Z}$ for odd $p$ and into $(\mathbb{Z}/2\mathbb{Z})^{2}$ for even $p$. In each case, the indicated group itself occurs as a subgroup.
\end{proposition}
\begin{proof}
Let $x=\sigma_{\{s_1\}}$, $y=\sigma_{\{s_2\}}$, and $z=\sigma_{S}$. These elements have order two. If $p$ is odd, opposition on $S$ exchanges $s_1,s_2$, so $zxz=y$ and $zyz=x$. Eliminating $y$ leaves $\langle x,z\mid x^2=z^2=1\rangle=D_\infty$. If $p$ is even, $z$ commutes with $x,y$ and there is no relation between $x,y$, giving $D_\infty\times\mathbb{Z}/2\mathbb{Z}$.

Writing $D_\infty=\langle r,x\mid x^2=1,xrx=r^{-1}\rangle$, every element can be written as $r^m$ or $xr^m$. Nontrivial $r^m$ have infinite order, and the product of two distinct reflections has infinite order. Thus, its finite subgroups have order at most two. If $p$ is even, let $F\leq D_\infty\times\mathbb{Z}/2\mathbb{Z}$ be finite, and let $\operatorname{pr}$ denote the projection onto $D_\infty$. Then $\operatorname{pr}(F)$ is a finite subgroup of $D_\infty$, so it has order at most two. Since $F\leq\operatorname{pr}(F)\times\mathbb{Z}/2\mathbb{Z}$, we conclude that $F$ embeds into $(\mathbb{Z}/2\mathbb{Z})^{2}$. For odd $p$, the subgroup $\langle x\rangle$ is isomorphic to $\mathbb{Z}/2\mathbb{Z}$. For even $p$, the subgroup $\langle x,z\rangle$ is isomorphic to $(\mathbb{Z}/2\mathbb{Z})^{2}$.
\end{proof}

\subsection{\texorpdfstring{Types $F$ and $H$}{Types F and H}}
Let $(W,S)$ have type $F_4$, $H_3$, or $H_4$. The longest element $\omega_S$ is central in $W$ \cite[Remark~3.1.2]{DJS03}. The nesting relations therefore imply that $\sigma_S$ is central in $C_W$.

To analyze the types $F_4$, $H_3$ and $H_4$, we use Lemma~\ref{lem:maxsupport} and Lemma~\ref{lem:central}. By Lemma~\ref{lem:central}, if $S\in\mathcal A$, then $ K_{\mathcal A}\cong K_{\mathcal A\setminus\{S\}}\times\mathbb{Z}/2\mathbb{Z}$. We then apply Lemma~\ref{lem:maxsupport} to $\mathcal A\setminus\{S\}$ to decompose its stabilizer over its maximal members.

\begin{figure}[htbp]
\centering
\begin{tikzpicture}[coxeter]

\begin{scope}
\node (aF) {$s_1$};
\node (bF) [right=of aF] {$s_2$};
\node (cF) [right=of bF] {$s_3$};
\node (dF) [right=of cF] {$s_4$};
\node [left=10mm of aF] {$F_4$};

\draw (aF.east)--(bF.west);
\draw (bF.east)--node[above] {$4$} (cF.west);
\draw (cF.east)--(dF.west);
\end{scope}

\begin{scope}[yshift=-1.5cm]
\node (aH3) {$s_1$};
\node (bH3) [right=of aH3] {$s_2$};
\node (cH3) [right=of bH3] {$s_3$};
\node [left=10mm of aH3] {$H_3$};

\draw (aH3.east)--node[above] {$5$} (bH3.west);
\draw (bH3.east)--(cH3.west);
\end{scope}

\begin{scope}[yshift=-3cm]
\node (aH4) {$s_1$};
\node (bH4) [right=of aH4] {$s_2$};
\node (cH4) [right=of bH4] {$s_3$};
\node (dH4) [right=of cH4] {$s_4$};
\node [left=10mm of aH4] {$H_4$};

\draw (aH4.east)--node[above] {$5$} (bH4.west);
\draw (bH4.east)--(cH4.west);
\draw (cH4.east)--(dH4.west);
\end{scope}

\end{tikzpicture}
\caption{Coxeter diagrams of types $F_4$, $H_3$, and $H_4$.}
\label{fig:typeFH}
\end{figure}

\begin{theorem}\label{thm:FH}
Let $F$ be a finite group.
\begin{enumerate}
\item $F\hookrightarrow C_{W(F_4)}$ if and only if $F\hookrightarrow(\mathbb{Z}/2\mathbb{Z})^{4}$.
\item $F\hookrightarrow C_{W(H_3)}$ if and only if $F\hookrightarrow(\mathbb{Z}/2\mathbb{Z})^{3}$.
\item $F\hookrightarrow C_{W(H_4)}$ if and only if $F$ embeds into $(\mathbb{Z}/2\mathbb{Z})^{4}$ or $\mathcal W_2\times\mathbb{Z}/2\mathbb{Z}$.
\end{enumerate}
\end{theorem}
\begin{proof}
We first show that every finite subgroup embeds into one of the groups specified for its type. Suppose that $F\hookrightarrow C_W$, where $W$ has type $F_4$, $H_3$, or $H_4$. By Theorem~\ref{thm:criterion}, there is a compatible family $\mathcal A$ such that $F\hookrightarrow K_{\mathcal A}$. It therefore suffices to prove the corresponding embedding for $K_{\mathcal A}$.

We first decompose $K_{\mathcal A}$ using the proper members of $\mathcal A$. Let $\mathcal D=\mathcal A\setminus\{S\}$, and let $\delta=1$ if $S\in\mathcal A$ and $\delta=0$ otherwise. Let $M_1,\ldots,M_k$ be the maximal members of $\mathcal D$, and define
\[
 \mathcal D_i=\{I\in\mathcal D:I\subseteq M_i\},
 \qquad
 K_i=\operatorname{Stab}_{C_{M_i}}(Q_{\mathcal D_i}).
\]
Since $\omega_S$ is central in all three types, Lemmas~\ref{lem:maxsupport} and~\ref{lem:central} give
\begin{equation}\label{eq:central-support-bound}
 K_{\mathcal A}\cong\prod_{i=1}^k K_i\times(\mathbb{Z}/2\mathbb{Z})^{\delta},
 \qquad
 K_i\leq C_{M_i}.
\end{equation}

To bound the rank when $K_{\mathcal A}$ is elementary abelian, we first bound its order. Every $h\in K_{\mathcal A}$ satisfies $h=h\cdot1\in V(Q_{\mathcal A})$. Thus, Lemmas~\ref{lem:vertices} and~\ref{lem:dimension} give
\begin{equation}\label{eq:stabilizer-order-bound}
 |K_{\mathcal A}|
 \leq |V(Q_{\mathcal A})|
 =2^{|\mathcal A|}
 \leq2^{|S|}.
\end{equation}
Consequently, if $K_{\mathcal A}$ is elementary abelian, its rank is at most $|S|$.

For type $F_4$, we show that $K_{\mathcal A}\hookrightarrow(\mathbb{Z}/2\mathbb{Z})^{4}$. Every proper connected subdiagram has type $A_1$, $A_2$, $B_2$, or $B_3$. By Theorem~\ref{thm:typeA} and Remark~\ref{rem:Bsmall}, each $K_i$ is elementary abelian. Therefore, the decomposition \eqref{eq:central-support-bound} shows that $K_{\mathcal A}$ is elementary abelian, and \eqref{eq:stabilizer-order-bound} bounds its rank by four.

For type $H_3$, we show that $K_{\mathcal A}\hookrightarrow(\mathbb{Z}/2\mathbb{Z})^{3}$. Every proper connected subdiagram has type $A_1$, $A_2$, or $I_2(5)$. By Theorem~\ref{thm:typeA} and Proposition~\ref{prop:dihedral}, each $K_i$ has order at most two. Thus, $K_{\mathcal A}$ is elementary abelian by \eqref{eq:central-support-bound}, and its rank is at most three by \eqref{eq:stabilizer-order-bound}.

For type $H_4$, we show that $K_{\mathcal A}$ embeds into $(\mathbb{Z}/2\mathbb{Z})^{4}$ or $\mathcal W_2\times\mathbb{Z}/2\mathbb{Z}$. Suppose first that some $M_i$ has three vertices. It is the only maximal member of $\mathcal D$: the remaining vertex of $S$ is adjacent to $M_i$, so no nonempty subset of it is orthogonal to $M_i$. Hence, \eqref{eq:central-support-bound} reduces to $K_{\mathcal A}\cong K_i\times(\mathbb{Z}/2\mathbb{Z})^{\delta}$. The subdiagram on $M_i$ has type $H_3$ or $A_3$. If it has type $H_3$, the preceding paragraph gives $K_i\hookrightarrow(\mathbb{Z}/2\mathbb{Z})^{3}$, so $K_{\mathcal A}\hookrightarrow(\mathbb{Z}/2\mathbb{Z})^{4}$. If it has type $A_3$, Theorem~\ref{thm:typeA} gives $K_i\hookrightarrow\mathcal W_2$, so $K_{\mathcal A}\hookrightarrow\mathcal W_2\times\mathbb{Z}/2\mathbb{Z}$.

If every $M_i$ has at most two vertices, its type is $A_1$, $A_2$, or $I_2(5)$. Each $K_i$ therefore has order at most two. It follows from \eqref{eq:central-support-bound} that $K_{\mathcal A}$ is elementary abelian. Its rank is at most four by \eqref{eq:stabilizer-order-bound}, giving $K_{\mathcal A}\hookrightarrow(\mathbb{Z}/2\mathbb{Z})^{4}$.

We next construct the groups in the statement as subgroups of the corresponding generalized cactus groups.

To realize $(\mathbb{Z}/2\mathbb{Z})^{4}$ in $C_{W(F_4)}$, we use $\sigma_{\{s_2\}}$, $\sigma_{\{s_2,s_3\}}$, $\sigma_{\{s_1,s_2,s_3\}}$, and $\sigma_S$. The corresponding subdiagrams are nested and have types $A_1$, $B_2$, $B_3$, and $F_4$, respectively. Their longest elements are central in the corresponding Coxeter groups, so these generators commute by the nesting relations. By Lemma~\ref{lem:vertices}, they generate a subgroup isomorphic to $(\mathbb{Z}/2\mathbb{Z})^{4}$.

To realize $(\mathbb{Z}/2\mathbb{Z})^{3}$ in $C_{W(H_3)}$, we use $\sigma_{\{s_1\}}$, $\sigma_{\{s_3\}}$, and $\sigma_S$. These generators commute because $\{s_1\}\perp\{s_3\}$ and $\omega_S$ is central. By Lemma~\ref{lem:vertices}, they generate a subgroup isomorphic to $(\mathbb{Z}/2\mathbb{Z})^{3}$.

To realize $(\mathbb{Z}/2\mathbb{Z})^{4}$ in $C_{W(H_4)}$, we use $\sigma_{\{s_1\}}$, $\sigma_{\{s_3\}}$, $\sigma_{\{s_1,s_2,s_3\}}$, and $\sigma_S$. The first two generators commute because $\{s_1\}\perp\{s_3\}$. The longest elements of $W_{\{s_1,s_2,s_3\}}$ and $W(H_4)$ are central, so the nesting relations show that all four generators commute. By Lemma~\ref{lem:vertices}, they generate a subgroup isomorphic to $(\mathbb{Z}/2\mathbb{Z})^{4}$.

Finally, we realize $\mathcal W_2\times\mathbb{Z}/2\mathbb{Z}$ in $C_{W(H_4)}$. Let $R=\{s_2,s_3,s_4\}$, whose subdiagram has type $A_3$. The construction in Lemma~\ref{lem:dyadic}, together with Proposition~\ref{prop:parabolic}, gives $L=\langle \sigma_{\{s_2\}},\sigma_{\{s_4\}},\sigma_R \rangle\cong\mathcal W_2$. The nesting relations and the centrality of $\omega_S$ show that $\sigma_S$ commutes with every element of $L$. Moreover, $g_{W(H_4)}(L)\leq W_R$, whereas $g_{W(H_4)}(\sigma_S)=\omega_S\notin W_R$. This implies that $\sigma_S\notin L$, and therefore,
\[
 \langle L,\sigma_S\rangle
 \cong L\times\langle\sigma_S\rangle
 \cong\mathcal W_2\times\mathbb{Z}/2\mathbb{Z}.
\]
\end{proof}

\section{\texorpdfstring
 {Proof of Theorem~\ref{thm:mainB}}
 {Proof of Theorem B}}
\label{sec:consequences}

In this section, we prove Theorem~\ref{thm:mainB}. As a consequence, we determine the possible orders of finite subgroups of $C_W$ when $S$ is finite.

\begin{proof}[Proof of Theorem~\ref{thm:mainB}]
We first prove Theorem~\ref{thm:mainB} for finite irreducible Coxeter systems. These follow from Theorems~\ref{thm:typeA}, \ref{thm:typeB}, \ref{thm:typeD}, \ref{thm:typeE}, \ref{thm:FH}, and Proposition~\ref{prop:dihedral}. For type $D_n$, we add factors $B_0=\mathbb{Z}/2\mathbb{Z}$ to the products $\prod_j B_{a_j}$ until the inequalities in Theorem~\ref{thm:typeD} become equalities. Each added factor increases $\sum_j2^{a_j}$ by one. Every group in condition~(3) of Theorem~\ref{thm:typeD} embeds into a group in the $D_n$ row of Table~\ref{tab:classification} by adding factors $B_0$. Conversely, every group in that row satisfies the inequalities in condition~(3). Thus, for a finite group $F$, condition~(3) of Theorem~\ref{thm:typeD} holds if and only if $F$ embeds into one of the groups in the $D_n$ row of Table~\ref{tab:classification}.

We next prove Theorem~\ref{thm:mainB} for finite reducible Coxeter systems. Let $(W,S)=\prod_{i=1}^r(W_i,S_i)$ be the irreducible decomposition, and let $X_i$ be the type of $(W_i,S_i)$. Suppose that $F$ is a finite group with $F\hookrightarrow C_W$. By Proposition~\ref{prop:product}, there are finite subgroups $F_i\leq C_{W_i}$ such that $F\hookrightarrow\prod_{i=1}^r F_i$. The irreducible classification gives an embedding $F_i\hookrightarrow H_i$ for some $H_i\in\mathscr L(X_i)$. Consequently,
\[
 F\hookrightarrow\prod_{i=1}^r F_i
 \hookrightarrow\prod_{i=1}^r H_i.
\]

Conversely, choose $H_i\in\mathscr L(X_i)$ for each $i$. The irreducible realization results give embeddings $H_i\hookrightarrow C_{W_i}$. Taking their product gives an embedding
\[
 \prod_{i=1}^r H_i
 \hookrightarrow\prod_{i=1}^r C_{W_i}
 \cong C_W.
\]
Thus, every such product is realized as a subgroup of $C_W$, and composing with an embedding $F\hookrightarrow\prod_{i=1}^r H_i$ gives $F\hookrightarrow C_W$. This proves the result for finite Coxeter systems.

We now prove Theorem~\ref{thm:mainB} for an arbitrary Coxeter system $(W,S)$. Let $F$ be a finite group. Suppose that condition~(1) holds, and identify $F$ with its image in $C_W$. By Theorem~\ref{thm:localization}, there exist a spherical subset $T\subseteq S$ and an element $a\in C_W$ such that $a^{-1}Fa\leq C_{W_T}.$ In particular, condition~(1) implies condition~(2). For every spherical subset $T$, the classification for finite Coxeter systems proved above applies to $(W_T,T)$. Thus, condition~(2) implies condition~(3). To prove that condition~(3) implies condition~(1), suppose that $F\hookrightarrow H$ for some spherical subset $T\subseteq S$ and some $H\in\mathscr L(T)$. Since $W_T$ is finite, the argument for finite Coxeter systems given above, applied to $(W_T,T)$, gives an embedding $H\hookrightarrow C_{W_T}$. Proposition~\ref{prop:parabolic} then gives
\[
 H\hookrightarrow C_{W_T}\hookrightarrow C_W.
\]
Composing with the given embedding $F\hookrightarrow H$ yields $F\hookrightarrow C_W$.
\end{proof}

When $S$ is finite, it is enough to use the maximal spherical subsets in Theorem~\ref{thm:mainB}. Indeed, every spherical subset $U\subseteq S$ is contained in a maximal spherical subset $T$, and Proposition~\ref{prop:parabolic} gives $C_{W_U}\hookrightarrow C_{W_T}$. When $S$ is infinite, maximal spherical subsets need not exist, and the orders of finite subgroups of $C_W$ can be unbounded.

We now determine the largest order of a finite subgroup when $S$ is finite. Let $s_2(m)$ denote the number of ones in the binary expansion of $m$. For a finite irreducible Coxeter type $X$, let $q(X)$ be the integer such that $2^{q(X)}$ is the largest order of a finite subgroup of $C_{W(X)}$. The classifications proved above give
\begin{equation}\label{eq:qdefinition}
 q(X)=
 \begin{cases}
  n+1-s_2(n+1),&X=A_n,\\
  5,&X=E_6,\\
  1,&X=I_2(p),\ p\text{ odd},\\
  \operatorname{rank}(X),&\text{otherwise}.
 \end{cases}
\end{equation}
For a finite Coxeter system $(V,R)$ with irreducible component types $X_1,\ldots,X_r$, put
\[
 q(V)=\sum_{i=1}^r q(X_i),
 \qquad q(1)=0.
\]

\begin{corollary}\label{cor:maximum}
Let $(W,S)$ be a Coxeter system with $|S|<\infty$. Then
\[
 \max\{|F|:F\leq C_W\text{ finite}\}
 =
 \max_{\substack{T\subseteq S\\
                  T\text{ maximal spherical}}}
 2^{q(W_T)}.
\]
The orders of finite subgroups of $C_W$ are exactly $2^j$ for
\[
 0\leq j\leq
 \max_{\substack{T\subseteq S\\
                  T\text{ maximal spherical}}}
 q(W_T).
\]
In particular, if $W$ is finite, the orders of finite subgroups of $C_W$ are exactly $2^j$ for $0\leq j\leq q(W)$.
\end{corollary}

\begin{proof}
For every spherical subset $T\subseteq S$, the largest order of a finite subgroup of $C_{W_T}$ is $2^{q(W_T)}$. Indeed, write $(W_T,T)=\prod_{i=1}^r(W_{T_i},T_i)$ for its irreducible decomposition, with component types $X_1,\ldots,X_r$. Under the isomorphism $C_{W_T}\cong\prod_{i=1}^r C_{W_{T_i}}$, a finite subgroup is contained in the product of its coordinate projections. By the definition of $q(X_i)$, its order is therefore at most $\prod_{i=1}^r 2^{q(X_i)}=2^{q(W_T)}.$ Taking the product of subgroups of largest order in the factors attains this bound.

We next bound the order of an arbitrary finite subgroup $F\leq C_W$. By Theorem~\ref{thm:localization}, a conjugate of $F$ is contained in $C_{W_U}$ for some spherical subset $U\subseteq S$. Since $S$ is finite, choose a maximal spherical subset $T$ containing $U$. Proposition~\ref{prop:parabolic} identifies $C_{W_U}$ with a subgroup of $C_{W_T}$, so
\[
 |F|\leq 2^{q(W_T)}
 \leq
 \max_{\substack{T'\subseteq S\\
                  T'\text{ maximal spherical}}}
 2^{q(W_{T'})}.
\]

To attain this bound, choose a maximal spherical subset $T_0$ for which $q(W_{T_0})$ is largest; such a subset exists because $S$ is finite. Choose a finite subgroup $H\leq C_{W_{T_0}}$ of order $2^{q(W_{T_0})}$. By Proposition~\ref{prop:parabolic}, $H$ embeds into $C_W$. Its image has order equal to the upper bound above, proving the formula for the largest order.

By Lemma~\ref{lem:pure}, every finite subgroup of $C_W$ is a $2$-group. The subgroup of largest order has subgroups of every $2$-power order dividing its order, giving all the orders asserted in the statement.

If $W$ is finite, then $S$ is the unique maximal spherical subset, giving the final assertion.
\end{proof}

By \eqref{eq:qdefinition}, $q(W_T)\leq|T|$ for every spherical subset $T$. Consequently, every finite subgroup of $C_W$ has order at most $2^{|S|}$. If $W$ is infinite, every spherical subset is properly contained in $S$, so this bound improves to $2^{|S|-1}$.

\begin{remark}
In general, the list of finite groups in Theorem~\ref{thm:mainB} cannot be replaced by a single finite group. For example, suppose that a finite group $G$ satisfies
\[
 F\hookrightarrow C_{W(E_7)}
 \quad\Longleftrightarrow\quad
 F\hookrightarrow G
\]
for every finite group $F$. Taking $F=G$ shows that $G\hookrightarrow C_{W(E_7)}$, so $|G|\leq128$ by Theorem~\ref{thm:typeE}. Both $(\mathbb{Z}/2\mathbb{Z})^{7}$ and $\mathcal W_2\times(\mathbb{Z}/2\mathbb{Z})^{4}$ must embed into $G$. Since both have order $128$, $G$ would have to be isomorphic to each of them. This is impossible, since the first group is abelian and the second is not.
\end{remark}

\section{Iterated wreath products and embeddings}
\label{sec:wreath}

For a group $A$ and an integer $r\geq1$, write
\[
 A\mathbin{\wr_{r}}\mathbb{Z}/2\mathbb{Z}=A^r\rtimes\mathbb{Z}/2\mathbb{Z}
\]
for the permutational wreath product whose action is given by
\[
 (a_1,\ldots,a_r)\longmapsto(a_r,\ldots,a_1)
\]
for the nontrivial element of $\mathbb{Z}/2\mathbb{Z}$. In particular, $A\mathbin{\wr_{1}}\mathbb{Z}/2\mathbb{Z}=A\times\mathbb{Z}/2\mathbb{Z}$. A \emph{tower} is a group $T=T_s$ equipped with a construction
\[
 T_0=\mathbb{Z}/2\mathbb{Z},\qquad
 T_i=T_{i-1}\mathbin{\wr_{r_i}}\mathbb{Z}/2\mathbb{Z}
 \quad(1\leq i\leq s),
\]
where $s\geq0$ and $r_i\geq1$.

We define the weight of this construction recursively by
\begin{equation}\label{eq:weight}
 u_0=1,\qquad
 u_i=
 \begin{cases}
  r_i u_{i-1},&r_i\geq2,\\
  u_{i-1}+1,&r_i=1,
 \end{cases}
 \qquad w(T)=u_s.
\end{equation}
At each stage, we construct a permutation action of $T_i$ on the set $\{1,\ldots,2u_i\}$. For $T_0=\mathbb{Z}/2\mathbb{Z}$, we use the transposition $(1\ 2)$ on $\{1,2\}$, giving $u_0=1$.

If $r_i\geq2$, partition $\{1,\ldots,2r_i u_{i-1}\}$ into $r_i$ consecutive intervals, each containing $2u_{i-1}$ elements. Let one copy of $T_{i-1}$ act on each interval. The additional involution reverses the order of the intervals while preserving the order within each interval. The set on which $T_i$ acts therefore has $2r_i u_{i-1}$ elements, giving $u_i=r_i u_{i-1}$.

If $r_i=1$, retain the action of $T_{i-1}$ on $\{1,\ldots,2u_{i-1}\}$ and extend it to fix $2u_{i-1}+1$ and $2u_{i-1}+2$. Let the additional factor $\mathbb{Z}/2\mathbb{Z}$ interchange these two elements and fix all the others. The resulting action uses $2u_{i-1}+2$ elements, giving $u_i=u_{i-1}+1$. Therefore, $w(T)$ is half the cardinality of the set used in this permutation construction.

\begin{lemma}\label{lem:towerdegree}
A tower of weight $u$ embeds into $S_{2u}$.
\end{lemma}

\begin{proof}
We prove by induction on $i$ that $T_i\hookrightarrow S_{2u_i}$. For $i=0$, this is the standard embedding $T_0=\mathbb{Z}/2\mathbb{Z}\hookrightarrow S_2$.

Suppose that $T_{i-1}\hookrightarrow S_{2u_{i-1}}$. If $r_i\geq2$, we use the following elementary observation: if $A\hookrightarrow S_m$ and $r\geq2$, then $A\mathbin{\wr_{r}}\mathbb{Z}/2\mathbb{Z}\hookrightarrow S_{rm}$. Indeed, let $A^r$ act on $r$ disjoint copies of $\{1,\ldots,m\}$, with the $j$th factor acting on the $j$th copy. Let the additional involution reverse the order of these copies, sending each element to the corresponding element in the other copy. Each factor acts injectively on its own set, and the additional involution exchanges the first and last sets. Thus, the resulting homomorphism $A\mathbin{\wr_{r}}\mathbb{Z}/2\mathbb{Z}\to S_{rm}$ is injective.

Applying this observation gives
\[
 T_i
 =T_{i-1}\mathbin{\wr_{r_i}}\mathbb{Z}/2\mathbb{Z}
 \hookrightarrow S_{2r_i u_{i-1}}
 =S_{2u_i}.
\]

If $r_i=1$, then $T_i=T_{i-1}\times\mathbb{Z}/2\mathbb{Z}$. By the induction hypothesis and the standard embedding $\mathbb{Z}/2\mathbb{Z}\hookrightarrow S_2$, letting the two factors act on disjoint sets gives
\[
 T_i
 \hookrightarrow S_{2u_{i-1}}\times S_2
 \hookrightarrow S_{2u_{i-1}+2}
 =S_{2u_i}.
\]
This completes the induction. Taking $i=s$ gives $T\hookrightarrow S_{2w(T)}$.
\end{proof}

In \cite[Conjecture~8.5]{Genevois25}, Genevois proposes an embedding criterion using a single tower with $\prod_i r_i\leq n/2$. The following theorem establishes a modified criterion using a finite product of towers whose total weight is at most $n/2$.

\begin{theorem}
\label{thm:weightedA}
Let $n\geq2$ and let $F$ be a finite group. Then $F\hookrightarrow J_n$ if and only if there exist towers $T_1,\ldots,T_k$ such that
\[
 F\hookrightarrow\prod_{j=1}^kT_j,\qquad
 \sum_{j=1}^k w(T_j)\leq\frac n2.
\]
\end{theorem}

\begin{proof}
Suppose that $F\hookrightarrow J_n$, and write $n=\sum_j2^{a_j}$ for its binary expansion. By Theorem~\ref{thm:typeA},
\[
 F\hookrightarrow\prod_j\mathcal W_{a_j}.
\]
We omit the factor $\mathcal W_0=\{1\}$ if it occurs. For $a\geq1$, the group $\mathcal W_a$ is obtained from $\mathbb{Z}/2\mathbb{Z}$ by applying the operation $A\mapsto A\mathbin{\wr_{2}}\mathbb{Z}/2\mathbb{Z}$ exactly $a-1$ times. It follows that the tower $\mathcal W_a$ has weight $2^{a-1}$. The total weight is
\[
 \sum_{a_j\geq1}2^{a_j-1}
 =\lfloor n/2\rfloor\leq n/2.
\]

Now suppose that $F\hookrightarrow\prod_{j=1}^kT_j$ and $\sum_jw(T_j)\leq n/2$. Choose disjoint consecutive intervals $\Omega_1,\ldots,\Omega_k\subseteq\{1,\ldots,n\}$ with $|\Omega_j|=2w(T_j)$. By Lemma~\ref{lem:towerdegree}, each $T_j$ embeds into $S_{\Omega_j}$. Taking the product of these embeddings gives $\prod_jT_j\hookrightarrow S_n$. This product is a $2$-group, so Sylow's theorem gives an embedding into $P_n$. Therefore, Theorem~\ref{thm:typeA} gives $F\hookrightarrow J_n$.
\end{proof}

We also obtain a weighted formulation for type $B$.

\begin{theorem}\label{thm:weightedB}
Let $n\geq1$ and let $F$ be a finite group. Then $F\hookrightarrow C_{W(B_n)}$ if and only if there exist groups $R_1,\ldots,R_k$, each either a tower or the trivial group, such that
\[
 F\hookrightarrow\prod_{j=1}^k(\mathbb{Z}/2\mathbb{Z}\times R_j),\qquad
 \sum_{j=1}^k w(R_j)\leq\frac n2,
\]
where towers have the weight in \eqref{eq:weight} and $w(1)=1/2$.
\end{theorem}

\begin{proof}
Suppose that $F\hookrightarrow C_{W(B_n)}$. By Theorem~\ref{thm:typeB}, there exist integers $a_1,\ldots,a_k\geq0$ such that
\[
 n=\sum_{j=1}^k2^{a_j},
 \qquad
 F\hookrightarrow\prod_{j=1}^k(\mathbb{Z}/2\mathbb{Z}\times\mathcal W_{a_j}).
\]
Take $R_j=\mathcal W_{a_j}$, using the tower construction above when $a_j\geq1$ and the convention $R_j=1$ when $a_j=0$. In both cases, $w(R_j)=2^{a_j-1}$, so $\sum_{j=1}^k w(R_j)=n/2$.

Now suppose that $F\hookrightarrow\prod_{j=1}^k(\mathbb{Z}/2\mathbb{Z}\times R_j)$ and $\sum_jw(R_j)\leq n/2$. Choose disjoint consecutive intervals $\Omega_j\subseteq\{1,\ldots,n\}$ with $|\Omega_j|=2w(R_j)$. Embed each tower $R_j$ into $S_{\Omega_j}$ using Lemma~\ref{lem:towerdegree}. If $R_j=1$, then $|\Omega_j|=1$ and $S_{\Omega_j}=\{1\}$. Identify the groups with their images and let
\[
 P=\prod_{j=1}^k R_j\leq S_n.
\]
Each $\Omega_j$ is nonempty and $P$-invariant, so $c(P)\geq k$. Consequently,
\[
 F\hookrightarrow\prod_{j=1}^k(\mathbb{Z}/2\mathbb{Z}\times R_j)
 \cong(\mathbb{Z}/2\mathbb{Z})^{k}\times P
 \hookrightarrow(\mathbb{Z}/2\mathbb{Z})^{c(P)}\times P.
\]
Each $R_j$ is a $2$-group, since it is either trivial or obtained from $\mathbb{Z}/2\mathbb{Z}$ by iterated direct powers and semidirect products with $\mathbb{Z}/2\mathbb{Z}$, so $P$ is a $2$-subgroup of $S_n$. Therefore, Theorem~\ref{thm:typeB} gives $F\hookrightarrow C_{W(B_n)}$.
\end{proof}

The inequality in Theorem~\ref{thm:weightedB} may be replaced by equality. Indeed, $n-2\sum_jw(R_j)$ is a nonnegative integer. Adjoin that many factors with $R_j=1$. Each contributes a factor $\mathbb{Z}/2\mathbb{Z}$ to the product and increases the total weight by $1/2$. The total weight then equals $n/2$, and the original product embeds into the enlarged product.

\subsection{Counterexamples to Genevois's conjecture}
Conjecture~8.5 in \cite{Genevois25} proposes one tower with $r_1\cdots r_s\leq n/2$. Under the stated convention $r_i\geq1$, a parameter equal to one adds a direct factor $\mathbb{Z}/2\mathbb{Z}$ without changing the product of parameters. For example, $n=2$ and $r_1=1$ allow $(\mathbb{Z}/2\mathbb{Z})^{2}$ on the proposed right-hand side, whereas $J_2\cong\mathbb{Z}/2\mathbb{Z}$. The additive cost in \eqref{eq:weight} addresses this issue.

The following proposition explains why Theorem~\ref{thm:weightedA} uses products of towers.

\begin{proposition}\label{prop:twenty}
Let $P_{20}$ be a Sylow $2$-subgroup of $S_{20}$. Then $P_{20}$ is not isomorphic to any tower. Consequently, $P_{20}$ does not embed into any tower $T$ satisfying $T\hookrightarrow J_{20}$.
\end{proposition}

\begin{proof}
Since $20=2^4+2^2$, Lemma~\ref{lem:sylow} gives $P_{20}\cong\mathcal W_4\times\mathcal W_2$. Suppose that a tower $T=T_s$ is isomorphic to $P_{20}$, with construction
\[
 T_0=\mathbb{Z}/2\mathbb{Z},\qquad
 T_i=T_{i-1}\mathbin{\wr_{r_i}}\mathbb{Z}/2\mathbb{Z}
 \quad(1\leq i\leq s).
\]

For a finite $2$-group $G$, let $d(G)$ denote its minimal number of generators, and put $e(G)=\log_2|G|$. Write $e_i=e(T_i)$ and $d_i=d(T_i)$. Then
\begin{equation}\label{eq:edrecursion}
 e_0=d_0=1,\qquad
 e_i=1+r_i e_{i-1},\qquad
 d_i=1+\left\lceil\frac{r_i}{2}\right\rceil d_{i-1}.
\end{equation}

To justify the formula for $d_i$, let $\Phi(A)$ denote the Frattini subgroup of a finite $2$-group $A$. The Burnside basis theorem gives
\[
 \Phi(A)=A^2[A,A],
 \qquad
 d(A)=\dim_{\mathbb F_2}(A/\Phi(A)),
\]
where $A^2=\langle a^2:a\in A\rangle$. In the Frattini quotient of $A\mathbin{\wr_{r}}\mathbb{Z}/2\mathbb{Z}$, the copies of $A/\Phi(A)$ in positions $j$ and $r-j+1$ are identified. The additional involution contributes an independent factor $\mathbb{Z}/2\mathbb{Z}$, so
\[
 \frac{A\mathbin{\wr_{r}}\mathbb{Z}/2\mathbb{Z}}{\Phi(A\mathbin{\wr_{r}}\mathbb{Z}/2\mathbb{Z})}
 \cong
 (A/\Phi(A))^{\lceil r/2\rceil}\times\mathbb{Z}/2\mathbb{Z}.
\]
This proves the formula for $d_i$.

We now determine the possible parameters $r_i$. The recurrence with $r_i=2$ gives $d(\mathcal W_a)=a$. Since $d$ is additive on direct products of finite $2$-groups, we obtain
\[
 e(P_{20})=15+3=18,\qquad
 d(P_{20})=4+2=6.
\]
Thus $(e_s,d_s)=(18,6)$.

Reading \eqref{eq:edrecursion} backwards, a step with parameter $r$ must have preceding pair
\[
 (e',d')
 =
 \left(
  \frac{e-1}{r},
  \frac{d-1}{\lceil r/2\rceil}
 \right),
\]
with both coordinates positive integers. Since $r_i\geq1$, the recurrence gives $d_i\geq d_{i-1}+1$, and hence $d_s\geq s+1$. Note that if $d_s=1$, then $s=0$ and $e_s=1$. If $d_s=2$, then $s=1$, and
\[
 2=d_1=1+\left\lceil\frac{r_1}{2}\right\rceil.
\]
Thus $r_1=1$ or $2$, giving $e_s=2$ or $3$.

At $(18,6)$, the only possible parameter is $r=1$, giving $(17,5)$. The possible preceding pairs for $(17,5)$ are
\[
 (16,4),\qquad (8,4),\qquad (4,2),\qquad (2,1).
\]
The last two are excluded by the restrictions above. The pair $(16,4)$ is also impossible: its possible preceding pairs are $(15,3)$ and $(3,1)$. The latter is excluded, while the only preceding pairs for $(15,3)$ are $(14,2)$ and $(7,2)$, both excluded. Thus, $(17,5)$ must be preceded by $(8,4)$, with parameter $r=2$.

The only preceding pair for $(8,4)$ is $(7,3)$. For $(7,3)$, the possibilities are $(6,2)$, $(3,2)$, and $(2,1)$, leaving only $(3,2)$. Finally, the possibilities for $(3,2)$ are $(2,1)$ and $(1,1)$, leaving only $(1,1)$. Consequently, the unique possible chain is
\[
 (18,6)\xrightarrow{1}(17,5)
 \xrightarrow{2}(8,4)
 \xrightarrow{1}(7,3)
 \xrightarrow{2}(3,2)
 \xrightarrow{2}(1,1).
\]
Reading the parameters in the construction order gives $(r_1,\ldots,r_s)=(2,2,1,2,1)$.

It remains to distinguish this tower from $P_{20}$. For the sequence above,
\[
 T\cong
 \bigl((\mathcal W_3\times\mathbb{Z}/2\mathbb{Z})\mathbin{\wr_{2}}\mathbb{Z}/2\mathbb{Z}\bigr)
 \times\mathbb{Z}/2\mathbb{Z}.
\]
For every nontrivial group $A$,
\[
 Z(A\mathbin{\wr_{2}}\mathbb{Z}/2\mathbb{Z})
 =
 \{(z,z):z\in Z(A)\}
 \cong Z(A).
\]
Starting from $\mathcal W_1=\mathbb{Z}/2\mathbb{Z}$, this gives $Z(\mathcal W_a)\cong\mathbb{Z}/2\mathbb{Z}$ for every $a\geq1$. Hence,
\[
 \begin{aligned}
 |Z(P_{20})|
 &=|Z(\mathcal W_4)|\,|Z(\mathcal W_2)|=4,\\
 |Z(T)|
 &=|Z(\mathcal W_3\times\mathbb{Z}/2\mathbb{Z})|\,|\mathbb{Z}/2\mathbb{Z}|=8.
 \end{aligned}
\]
Their centers have different orders, contradicting $T\cong P_{20}$. Therefore, $P_{20}$ is not isomorphic to any tower.

Finally, suppose that $P_{20}\hookrightarrow T\hookrightarrow J_{20}$ for some tower $T$. By Theorem~\ref{thm:typeA}, we also have $T\hookrightarrow P_{20}$. Comparing orders gives $|T|=|P_{20}|$, so $T\cong P_{20}$, contradicting the result just proved. This establishes the final assertion.
\end{proof}

\end{document}